\documentclass[sn-mathphys-num]{sn-jnl}
\usepackage[utf8]{inputenc}
\usepackage{amsmath,color}
\usepackage{amsfonts}
\usepackage{epsfig}
\usepackage{amssymb,fullpage}
\numberwithin{equation}{section}
\usepackage{amsthm}

\newtheorem{lemma}{Lemma}
\theoremstyle{remark}
\newtheorem{example}{Example}
\numberwithin{example}{section}
\begin{document}

\title[Chemical reactions systems with non-hyperbolic critical points]{Classification of strongly non-hyperbolic critical points of planar second-order chemical reaction systems}

\author[*]{\fnm{Norm} \sur{Yeung}}\email{norm.yeung@maths.ox.ac.uk}

\author[*]{\fnm{Radek} \sur{Erban}}\email{radek.erban@maths.ox.ac.uk}

\affil[*]{\orgdiv{Mathematical Institute}, \orgname{University of Oxford}, \orgaddress{\street{Radcliffe Observatory Quarter,
Woodstock Road}, \city{Oxford}, \postcode{OX2 6GG}, \country{United Kingdom}}}

\abstract{The dynamics of chemical reaction systems with two chemical species can be described (under mass-action kinetics) by planar autonomous systems of ordinary differential equations (ODEs) with right-hand sides containing polynomials. While their behaviour close to hyperbolic critical points is well understood, their dynamics has not been fully characterized for strongly non-hyperbolic critical points, where the linearization close to the equilibrium is given by the zero matrix. In this paper, the behaviour of such chemical reaction systems is investigated. Considering second-order chemical reaction networks with a positive equilibrium point, we show that only nine geometric equivalence classes are possible out of the sixteen classes for general ODE systems. Examples of chemical reaction networks in each class are presented. All sixteen geometric equivalence classes can be realized by chemical systems if the non-hyperbolic equilibrium point is located at the origin.}


\maketitle

\section{Introduction}

\label{sec1}

Chemical reaction networks provide a natural framework for describing a broad range of dynamical phenomena~\cite{Feinberg:2019:FCR,Erban:2020:SMR}. In recent years, the problem of identifying minimal chemical reaction systems exhibiting prescribed dynamical behaviour has been studied systematically, with particular emphasis on oscillatory dynamics~\cite{Banaji:2024:OTR,Boros:2022:LCM,Boros:2023:SMB,Boros:2024:OPD}, multistability~\cite{Wilhelm:2009:SCR,Craciun:2006:UBC, Joshi:2013:AMC,Joshi:2017:SRN,Conradi:2019:MBN,Tang:2021:MSR,Jiao:2025:MSZ}, and specific bifurcation structures~\cite{Wilhelm:1995:SCR,Wilhelm:1996:MAS,Plesa:2016:CRS,Plesa:2017:TMS,Banaji:2023:SBM,Banaji:2023:SBM2,Banaji:2024:BPQ}. Under mass-action kinetics, such questions are closely related to the study of dynamical properties of systems of ordinary differential equations (ODEs) with polynomial right-hand sides~\cite{Erban:2023:CSL,Craciun:2025:PCR}. Near a hyperbolic equilibrium, the local dynamics are determined, up to topological equivalence, by the linearization, as a consequence of the Hartman–Grobman theorem and the stable manifold theorem~\cite{Wiggins:2003:IAN}. By contrast, non-hyperbolic equilibria generally require a more refined nonlinear analysis. The situation is particularly intricate for strongly non-hyperbolic equilibria, by which we mean equilibria whose linearization is the zero matrix; their neighbourhoods may exhibit a rich variety of local phase portraits involving elliptic, hyperbolic, and parabolic sectors~\cite{Perko:2013:DED}.

For planar quadratic systems, the local phase portraits near such strongly non-hyperbolic equilibria were classified by Lyagina~\cite{Lyagina:1951:ICE}, who identified sixteen geometric equivalence classes. This classification was subsequently confirmed by Newton~\cite{Newton:1978:TDH}, who also related it to the alternative classification introduced by Markus~\cite{Markus:1960:QDE}. In the present paper, we address the corresponding classification problem for chemical reaction systems. Even in the quadratic case, the classifications of Lyagina~\cite{Lyagina:1951:ICE}, Markus~\cite{Markus:1960:QDE}, and Newton~\cite{Newton:1978:TDH} cannot be applied directly, since the polynomial vector fields arising from chemical reaction networks are subject to additional constraints on their coefficients. Our aim is therefore to determine which local phase portraits are compatible with these chemical constraints and to classify the resulting dynamics near strongly non-hyperbolic equilibria.

We consider chemical reaction networks with two chemical species $X$ and $Y$ which are subject to $m \in {\mathbb N}$ chemical reactions
\begin{equation}
\alpha_i X \, + \, \beta_i Y 
\;
\mathop{\longrightarrow}^{k_i}
\;
\gamma_i X \, + \, \delta_i Y \, ,
\qquad
\mbox{for}
\quad
i = 1, 2, \dots, m,
\label{crn}    
\end{equation}
where $\alpha_i \in {\mathbb N}_0,$ $\beta_i \in {\mathbb N}_0$, $\gamma_i \in {\mathbb N}_0$ and 
$\delta_i \in {\mathbb N}_0$ are nonnegative integers, called {\it stoichiometric coefficients}, 
satisfying
\begin{equation}
\alpha_i+\beta_i \le 2,
\label{secondorder}
\end{equation}
{\it i.e.} each chemical reaction is at most of the second-order (containing reactions of at most two molecules). The corresponding reaction {\it rate constant}, $k_i$, has physical units of [time]$^{-1}$[concentration]$^{1-\alpha_i-\beta_i}$~\cite{Craciun:2025:PCR}. However, in what follows, we assume that all chemical models have been non-dimensionalized, {\it i.e.} the rate constant $k_i$ is assumed to be a positive real number for $i=1,2,\dots,m$. Moreover, to avoid some degenerate cases, we assume that the same four-tuple $(\alpha_i, \beta_i,\gamma_i,\delta_i)$ does not occur more than once in the set of $m$ chemical reactions~(\ref{crn}) and we have $(\alpha_i, \beta_i) \ne (\gamma_i,\delta_i)$, {\it i.e.} in each reaction step at least one of the two chemical species $X$ and $Y$ changes. To get non-trivial planar systems, we also assume that both species take part in at least one reaction in the  set of $m$ chemical reactions~(\ref{crn}).

Assuming mass-action kinetics, the time evolution of the chemical reaction network~(\ref{crn}) is given by the reaction rate equations which is a planar polynomial ODE system in the following form~\cite{Feinberg:2019:FCR,Erban:2020:SMR} 
\begin{eqnarray}
\frac{\mbox{{\rm d}}x}{\mbox{{\rm d}}t}
& = &
\sum_{i=1}^m 
k_i \, (\gamma_i-\alpha_i) \, x^{\alpha_i} \, y^{\beta_i} \, , 
\label{general_chemical_system_x}
\\
\frac{\mbox{{\rm d}}y}{\mbox{{\rm d}}t}
& = &
\sum_{i=1}^m k_i \, (\delta_i-\beta_i) \, x^{\alpha_i} \, y^{\beta_i}   \, .
\label{general_chemical_system_y}
\end{eqnarray}
Using assumption~(\ref{secondorder}), this ODE system can be rewritten as
\begin{eqnarray}
\frac{\mbox{{\rm d}}x}{\mbox{{\rm d}}t}
\,& = &\,
a_{0,0}
\,+\,
a_{1,0} \, x
\,+\,
a_{0,1} \, y
\,+\,
a_{2,0} \, x^2
\,+\,
a_{1,1} \, x \, y
\,+\,
a_{0,2} \, y^2 \,,
\label{quadraticchemical1}
\\
\frac{\mbox{{\rm d}}y}{\mbox{{\rm d}}t}
\,& = &\,
b_{0,0}
\,+\,
b_{1,0} \, x
\,+\,
b_{0,1} \, y
\,+\,
b_{2,0} \, x^2
\,+\,
b_{1,1} \, x \, y
\,+\,
b_{0,2} \, y^2 \,,
\label{quadraticchemical2}
\end{eqnarray}
where $a_{ij}$, $b_{ij}$ are real numbers satisfying~\cite{Craciun:2025:PCR}
\begin{equation}
a_{0,0} \ge 0, \qquad
a_{0,1} \ge 0, \qquad
a_{0,2} \ge 0, \qquad
b_{0,0} \ge 0, \qquad
b_{1,0} \ge 0 \qquad
\mbox{and}
\qquad
b_{2,0} \ge 0\,.
\label{condchem}    
\end{equation}
The condition~(\ref{condchem}) is both a necessary and sufficient condition for the ODE system~(\ref{quadraticchemical1})--(\ref{quadraticchemical2}) to represent a chemically reacting system~\cite[Lemma 1]{Craciun:2025:PCR}. In particular, the classification of the dynamical behaviour of chemical reaction networks is equivalent to classifying the dynamical behaviour of ODEs in the form~(\ref{quadraticchemical1})--(\ref{quadraticchemical2}) satisfying condition~(\ref{condchem}). In this paper, we will focus on positive equilibrium points in Section~\ref{sec3}, because they provide the most restrictions on observable dynamics. If the equilibrium point is located at the origin, then the phase plane of any of the sixteen equivalence classes of Lyagina~\cite{Lyagina:1951:ICE} can be realized by chemical systems, although we are only interested in their restrictions to the non-negative quadrant. In what follows, we denote the positive and non-negative quadrants by 
\begin{equation}
\Omega \;=\;
\big\{ (x,y) \in {\mathbb R}^2 \; \mbox{where} \; x > 0, \; y > 0 \big\}
\qquad
\mbox{and}
\qquad
\overline{\Omega} \;=\;
\big\{ (x,y) \in {\mathbb R}^2 \; \mbox{where} \; x \ge 0, \; y \ge 0 \big\}\,.
\label{posquadrant}    
\end{equation}
The paper is organized as follows. In Section~\ref{sec2}, we start our discussion with the strongly non-hyperbolic equilibrium point located at the origin. This is followed by the classification of phase portraits close to the non-hyperbolic positive equilibria of chemical systems in Section~\ref{sec3}. Since different chemical reaction networks can correspond to the same ODE system~\cite{Plesa:2018:NCM}, a related question is to identify the simplest possible chemical reaction networks within each geometric equivalence class. We discuss this point further in our concluding Section~\ref{sec4}.

\section{Strongly non-hyperbolic equilibrium point at the origin}

\label{sec2}

We first investigate the behaviour of ODEs~(\ref{quadraticchemical1})--(\ref{quadraticchemical2}) close to a non-hyperbolic critical point at the origin $[0,0]$, where the linearization close to the equilibrium is given by the zero matrix. Then we have 
$$
a_{0,0} \, = \, 0\,, \qquad
a_{1,0} \, = \, 0\,, \qquad
a_{0,1} \, = \, 0\,, \qquad
b_{0,0} \, = \, 0\,, \qquad
b_{1,0} \, = \, 0\,, \qquad
b_{0,1} \, = \, 0\,, \qquad
$$
and the ODEs~(\ref{quadraticchemical1})--(\ref{quadraticchemical2}) read as follows
\begin{eqnarray}
\frac{\mbox{{\rm d}}x}{\mbox{{\rm d}}t}
\,& = &\,
a_{2,0} \, x^2
\,+\,
a_{1,1} \, x \, y
\,+\,
a_{0,2} \, y^2 \,,
\label{quadraticchemical1hom}
\\
\frac{\mbox{{\rm d}}y}{\mbox{{\rm d}}t}
\,& = &\,
b_{2,0} \, x^2
\,+\,
b_{1,1} \, x \, y
\,+\,
b_{0,2} \, y^2 \,,
\label{quadraticchemical2hom}
\end{eqnarray}
with conditions~(\ref{condchem}) simplifying to
\begin{equation}
a_{0,2} \ge 0 \qquad
\mbox{and}
\qquad
b_{2,0} \ge 0\,.
\label{condchemhom}    
\end{equation}
In general, two-dimensional homogeneous quadratic systems have sixteen geometric equivalence classes~\cite{Lyagina:1951:ICE,Newton:1978:TDH}. In our case, we restrict to systems satisfying inequalities~(\ref{condchemhom}). However, this restriction does not reduce the number of geometric equivalence classes. In particular, we can construct chemical systems with elliptic, parabolic and hyperbolic sectors as shown in the following two examples. We are only interested in their dynamics in the non-negative quadrant $\overline{\Omega}$ given by~(\ref{posquadrant}). 

\vskip 2mm

\begin{example}
\label{example1}
Consider the following system of $m=3$ chemical reactions:
\begin{equation}
2 X \,
\;
\mathop{\longrightarrow}^{1}
\;
3 X \, ,
\qquad
X \, + \, Y 
\;
\mathop{\longrightarrow}^{3}
\;
2 Y \, 
\qquad
\mbox{and}
\qquad
2 Y 
\;
\mathop{\longrightarrow}^{1}
\;
Y \, .
\label{crnex1}    
\end{equation}
Then the ODEs~(\ref{quadraticchemical1hom})--(\ref{quadraticchemical2hom}) reduce to
\begin{eqnarray}
\frac{\mbox{{\rm d}}x}{\mbox{{\rm d}}t}
\,& = &\,
x^2
\,-\,
3 \, x \, y \,,
\label{quadraticchemical1homex1}
\\
\frac{\mbox{{\rm d}}y}{\mbox{{\rm d}}t}
\,& = &\,
3 \, x \, y
\,-\,
y^2 \,,
\label{quadraticchemical2homex1}
\end{eqnarray}
which has an elliptic sector in the positive quadrant $\Omega$. Given the initial condition $(x(0),y(0)) \in \Omega$, we define constant $c=x(0) \, y(0) \big/ (x(0)+y(0))^{4}.$ Then the solution to the ODEs~(\ref{quadraticchemical1homex1})--(\ref{quadraticchemical2homex1}) satisfies
$$
x \, y \, = \, c \, (x+y)^4\,.
$$
In particular, we have
\begin{equation}
\lim_{t \to \pm \infty}
[x(t),y(t)]
\,=\,
[0,0].
\label{homorbitcond}
\end{equation}
All trajectories in the positive quadrant $\Omega$ are homoclinic orbits as shown in Figure~\ref{figure1}(a). 
\end{example}

\vskip 2mm

\begin{figure}
\rule{0pt}{1pt}
\vskip 1mm
\rule{0pt}{1pt}
\hskip -1.5cm
\epsfig{file=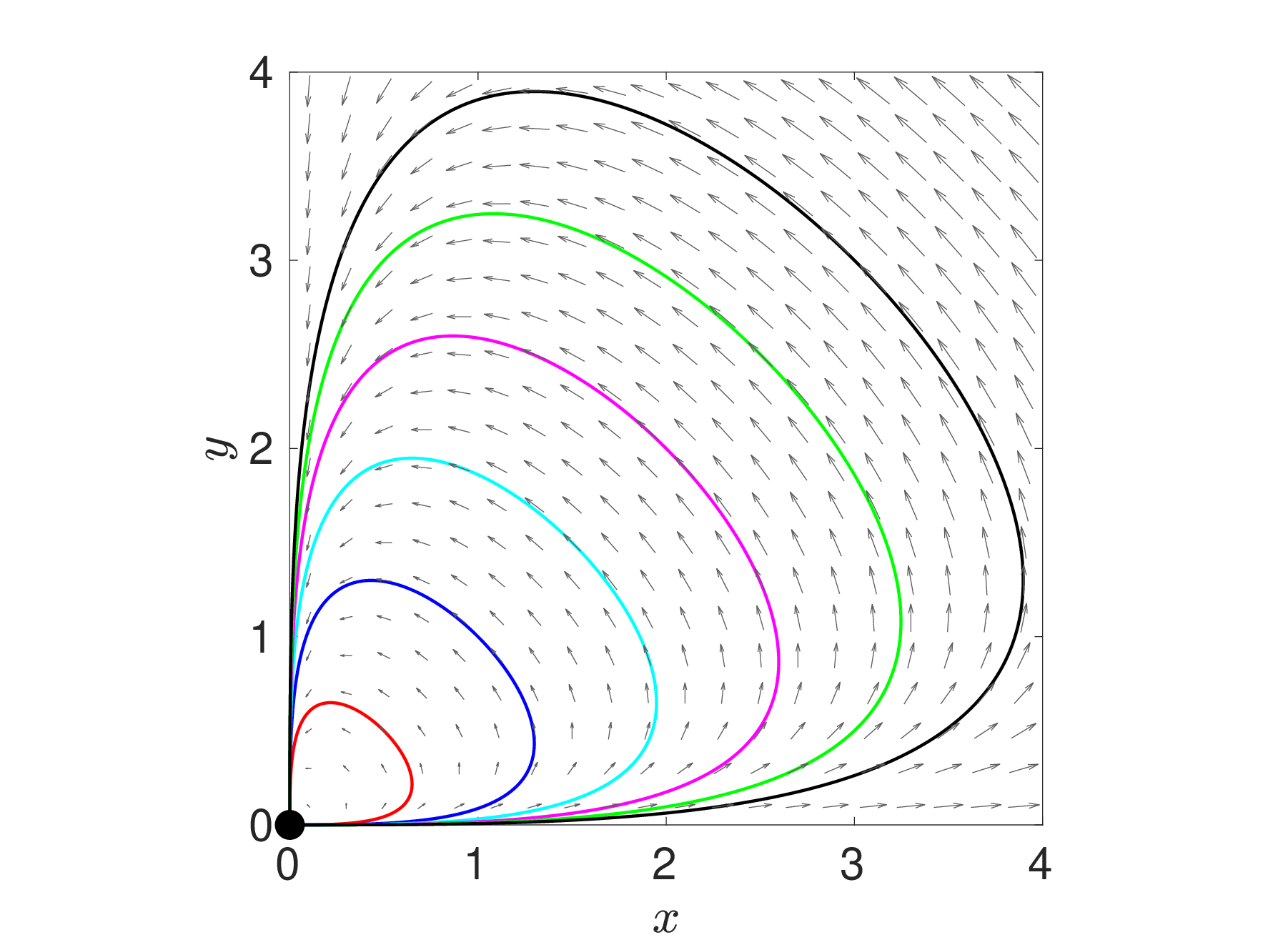,height=7.4cm}
\hskip -2cm
\epsfig{file=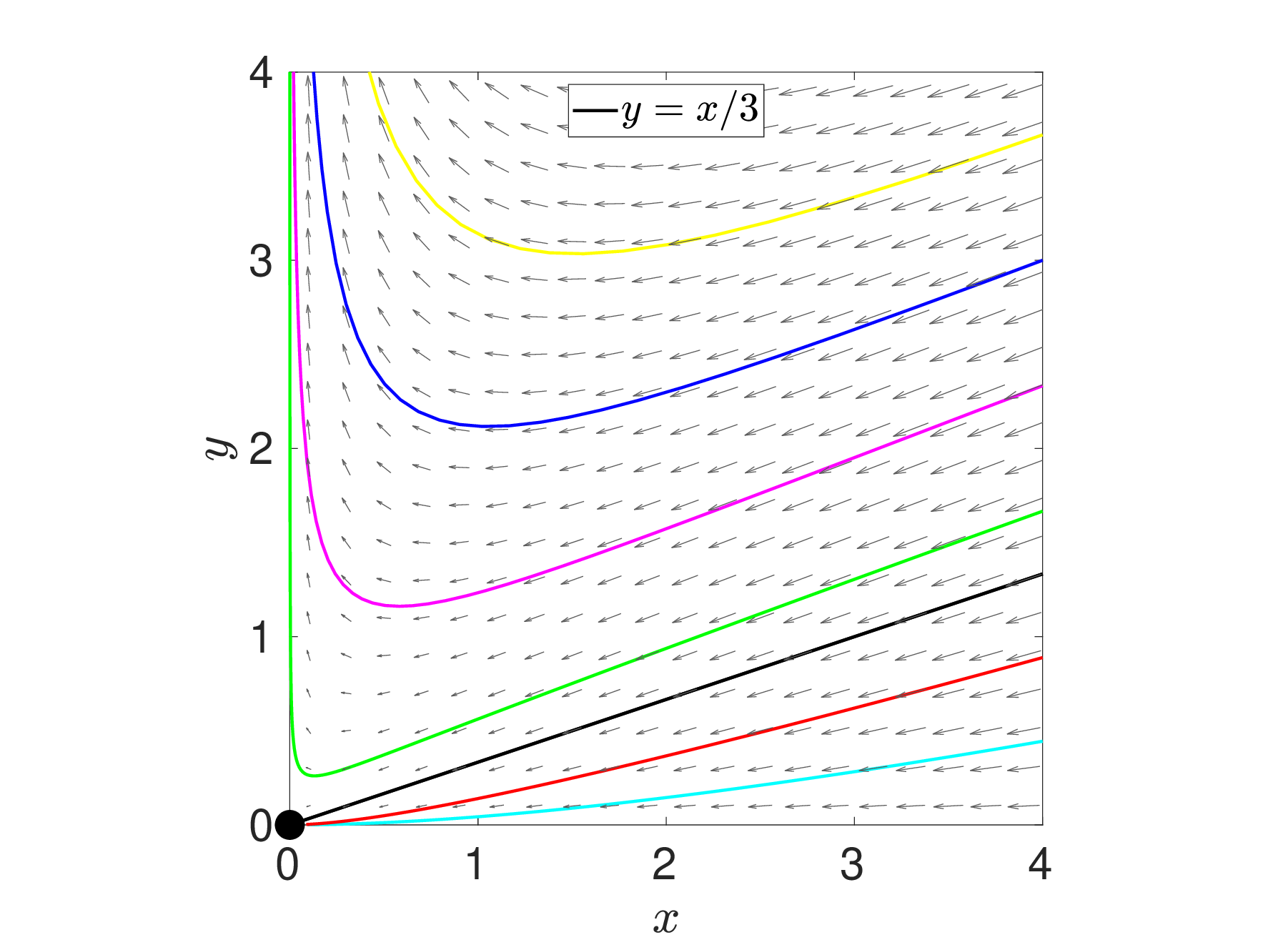,height=7.4cm}
\vskip -7.6cm
(a) \hskip 76mm (b)
\vskip 7.1cm
\caption{{\it {\rm (a)} The phase plane of the ODE system~$(\ref{quadraticchemical1homex1})$--$(\ref{quadraticchemical2homex1})$ describing the chemical reaction network~$(\ref{crnex1})$. We plot trajectories in the positive quadrant $\Omega$, which pass through the points $(x_0,x_0)$ for $x_0 \in \{ 0.5, 1, 1.5, 2, 2.5, 3\}$. The trajectories are homoclinic orbits satisfying~$(\ref{homorbitcond})$, illustrating that the positive quadrant is an elliptic sector. 
\hfill\break {\rm (b)} The phase plane of the ODE system~$(\ref{quadraticchemical1homex2})$--$(\ref{quadraticchemical2homex2})$  describing the chemical reaction network~$(\ref{crnex2})$. We plot illustrative trajectories in both the parabolic sector $\{ (x,y) \in \mathbb{R}^2 \,|\, 0 < y < x/3 \} $ and the hyperbolic sector $\{ (x,y) \in \mathbb{R}^2 \,|\, 0 < x/3 < y \} $. 
}}
\label{figure1}
\end{figure}

\begin{example}
\label{example2}
Consider the following system of $m=3$ chemical reactions:
\begin{equation}
2 X \,
\;
\mathop{\longrightarrow}^{1}
\;
X \, ,
\qquad
X \, + \, Y 
\;
\mathop{\longrightarrow}^{2}
\;
\emptyset \, 
\qquad
\mbox{and}
\qquad
2 Y 
\;
\mathop{\longrightarrow}^{1}
\;
3 Y \, .
\label{crnex2}    
\end{equation}
Then the ODEs~(\ref{quadraticchemical1hom})--(\ref{quadraticchemical2hom}) reduce to
\begin{eqnarray}
\frac{\mbox{{\rm d}}x}{\mbox{{\rm d}}t}
\,& = &\,
-x^2
\,-\,
2 \, x \, y \,,
\label{quadraticchemical1homex2}
\\
\frac{\mbox{{\rm d}}y}{\mbox{{\rm d}}t}
\,& = &\,
-2 \, x \, y
\,+\,
y^2 \,,
\label{quadraticchemical2homex2}
\end{eqnarray}
which has one parabolic sector and one hyperbolic sector in the positive quadrant $\Omega$. The boundary between these two sectors is given by the invariant line $y=x/3$ which is plotted as the black line in Figure~\ref{figure1}(b). The first integral of the system is $x (x-3y)^5 y^{-3} = \mbox{constant}.$ In particular, starting with the initial condition in the parabolic sector, we have
$$
\lim_{t \to \infty}
[x(t),y(t)]
\,=\,
[0,0]
\qquad
\mbox{for}
\qquad
y(0) \le x(0)/3,
$$
while we have $x(t) \to 0$ and $y(t) \to \infty$ as $t \to \infty$ for any trajectory starting in the hyperbolic sector $y(0) > x(0)/3.$
\end{example}

\vskip 2mm

\noindent
Examples~\ref{example1} and~\ref{example2} show that the phase planes of chemical systems with a strongly non-hyperbolic critical point at the origin can contain elliptic, parabolic, and hyperbolic sectors. Moreover, it is easy to verify that all sixteen of Lyagina’s equivalence classes~\cite{Lyagina:1951:ICE,Newton:1978:TDH} can be realized by such chemical systems. The situation is more delicate when the strongly non-hyperbolic equilibrium point has positive coordinates, as studied in the next Section~\ref{sec3}. In particular, we show that classes containing elliptic sectors, such as the one depicted in Figure~\ref{figure1}(a), cannot be realized by chemical systems when the non-hyperbolic critical point has positive coordinates.

\section{Positive strongly non-hyperbolic equilibrium point}

\label{sec3}

We next investigate the behaviour of ODEs~(\ref{quadraticchemical1})--(\ref{quadraticchemical2}) close to a non-hyperbolic positive critical point, where the linearization close to the equilibrium is given by the zero matrix. Without loss of generality, we assume that the non-hyperbolic positive critical point is at $[1,1]$. This means that we have
\begin{eqnarray}
0
\,& = &\,
a_{0,0}
\,+\,
a_{1,0}
\,+\,
a_{0,1}
\,+\,
a_{2,0}
\,+\,
a_{1,1}
\,+\,
a_{0,2}\,,
\label{eqpoint1}
\\
0
\,& = &\,
b_{0,0}
\,+\,
b_{1,0}
\,+\,
b_{0,1}
\,+\,
b_{2,0}
\,+\,
b_{1,1}
\,+\,
b_{0,2} \,,
\label{eqpoint2}
\end{eqnarray}
and the Jacobian at the critical point $[1,1]$ is the zero matrix, which implies
\begin{equation}
\begin{pmatrix}
a_{1,0}
+
2 \, a_{2,0} 
+
a_{1,1} \,
&
a_{0,1}
+
a_{1,1}
+
2 \, a_{0,2} 
\\
b_{1,0}
+
2\,b_{2,0} 
+
b_{1,1}
\,
&
b_{0,1}
+
b_{1,1}
+
2\,b_{0,2}
\end{pmatrix}
\,=\,
\begin{pmatrix}
0 & 0 \\
0 & 0
\end{pmatrix}.
\label{zeromatrixcondition}
\end{equation}
The assumptions~(\ref{eqpoint1})--(\ref{zeromatrixcondition}) help us to evaluate coefficients $a_{0,0}$, $a_{1,0}$, $a_{0,1}$, $b_{0,0}$, $b_{1,0}$ and $b_{0,1}$ in terms of coefficients $a_{2,0}$, $a_{1,1}$, $a_{0,2}$, $b_{2,0}$, $b_{1,1}$ and $b_{0,2}$. We obtain
\begin{eqnarray}
a_{0,0}
\,& = &\,
a_{2,0}
\,+\,
a_{1,1}
\,+\,
a_{0,2}
\label{othcoef1}
\\
b_{0,0}
\,& = &\,
b_{2,0}
\,+\,
b_{1,1}
\,+\,
b_{0,2} 
\label{othcoef2}
\\
a_{1,0}
\,& = &\,
\,-\,
2 \, a_{2,0} \,
\,-\,
a_{1,1}
\label{othcoef3}
\\
a_{0,1}
\,& = &\,
\,-\,
a_{1,1}
\,-\,
2 \, a_{0,2} 
\label{othcoef4}
\\
b_{1,0}
\,& = &\,
\,-\,
2\,b_{2,0} 
\,-\,
b_{1,1}
\label{othcoef5}
\\
b_{0,1}
\,& = &\,
\,-\,
b_{1,1}
\,-\,
2\,b_{0,2}\,.
\label{othcoef6}
\end{eqnarray}
Using relations~(\ref{othcoef1})--(\ref{othcoef6}), we conclude that we have six free real parameters $a_{2,0}$, $a_{1,1}$, $a_{0,2}$, $b_{2,0}$, $b_{1,1}$ and $b_{0,2}$, and conditions~(\ref{condchem}) can be written in the form
\begin{eqnarray}
a_{2,0}
\,+\,
a_{1,1}
\,+\,
a_{0,2}
\,\ge\, 
0 \,,
&\qquad&
b_{2,0}
\,+\,
b_{1,1}
\,+\,
b_{0,2} 
\,\ge\, 
0 \,,
\label{ineq1}
\\
\,-\,
a_{1,1}
\,\ge\, 
2 \, a_{0,2} 
\,\ge\, 
0 \,,
&\qquad&
\,-\,
b_{1,1}
\,\ge\, 
2\,b_{2,0} 
\,\ge\, 
0 \,.
\label{ineq2}
\end{eqnarray}
We have the following lemma.

\vskip 6mm

\begin{lemma}
Let $a_{2,0}$, $a_{1,1}$, $a_{0,2}$, $b_{2,0}$, $b_{1,1}$ and $b_{0,2}$ be real parameters satisfying inequalities~$(\ref{ineq1})$--$(\ref{ineq2})$. Consider the planar ODE system $(\ref{quadraticchemical1})$--$(\ref{quadraticchemical2})$, where parameters $a_{0,0}$, $a_{1,0}$, $a_{0,1}$, $b_{0,0}$, $b_{1,0}$ and $b_{0,1}$ are given by~$(\ref{othcoef1})$--$(\ref{othcoef6})$. Then $[1,1]$ is a non-hyperbolic critical point of the system, with the linearization given by the zero matrix. Moreover, introducing new variables
\begin{equation}
u \, = \, x - 1
\qquad
\mbox{and}
\qquad
v \, = \, y - 1\,,
\label{uvvariables}
\end{equation}
the planar ODE system $(\ref{quadraticchemical1})$--$(\ref{quadraticchemical2})$ transforms to
\begin{eqnarray}
\frac{\mbox{{\rm d}}u}{\mbox{{\rm d}}t}
\,& = &\,
a_{2,0} \, u^2
\,+\,
a_{1,1} \, u \, v
\,+\,
a_{0,2} \, v^2 \,,
\label{homODE1}
\\
\frac{\mbox{{\rm d}}v}{\mbox{{\rm d}}t}
\,& = &\,
b_{2,0} \, u^2
\,+\,
b_{1,1} \, u \, v
\,+\,
b_{0,2} \, v^2 \,.
\label{homODE2}
\end{eqnarray}
\label{lemma1}
\end{lemma}

\begin{proof}
Conditions~~(\ref{eqpoint1})--(\ref{zeromatrixcondition}) for the non-hyperbolic critical point at $[1,1]$, with the linearization given by the zero matrix, are satisfied by our choice of coefficients~(\ref{othcoef1})--(\ref{othcoef2}). Substituting the change of variables (\ref{uvvariables}) into~(\ref{quadraticchemical1})--(\ref{quadraticchemical2}), we obtain the ODE system~(\ref{homODE1})--(\ref{homODE2}).
\end{proof}

\vskip 5mm

\begin{table}[t]
\centering
\begin{tabular}{|c|c|c|ccccccccccc|}
\hline
\!\!Lyagina\!\!&
Newton  &  
Markus & \multicolumn{11}{c|}{example chemical system with $a_{02}=0$} \\
\cite{Lyagina:1951:ICE}
& \cite{Newton:1978:TDH} & \cite{Markus:1960:QDE} 
& \!\!\!$a_{2,0}$\!\!\!\! 
& \!\!\!\!$a_{1,1}$\!\!\!\! 
& \!\!\!\!$b_{2,0}$\!\!\!\! 
& \!\!\!\!$b_{1,1}$\!\!\!\! 
& \!\!\!\!$b_{0,2}$\!\!\!\! 
& \!\!\!\!$a_{1,0}$\!\!\!\! 
& \!\!\!\!$a_{0,1}$\!\!\!\! 
& \!\!\!\!$b_{1,0}$\!\!\!\!
& \!\!\!\!$b_{0,1}$\!\!\!\! 
& \!\!\!\!$a_{0,0}$\!\!\!\! 
& \!\!\!$b_{0,0}$\!\!\! \\
\hline
\rule{0pt}{2.8mm}1 & $H$ 
& $+ \leftarrow - \rightarrow$ 
& 6 & -6 & 2 & -4 & \!\!39\!\!
& -6 & 6 & 0 & \!\!-74\!\! & 0 & \!\!37\!\!
\\ \hline
\rule{0pt}{2.8mm}4 & $P_2H$ 
& $+ \leftarrow - \rightarrow$ 
& \!\!24\!\! & 0 & \!\!58\!\! & \!\!\!\!-116\!\!\!\! & 100 
& \!\!-48\!\! & 0 & 0 & \!\!-84\!\! & \!\!24\!\! & \!\!42\!\! \\ \hline
\rule{0pt}{2.8mm}5 & $H, P_i-P_j$ 
& $+ \rightarrow + \leftarrow - \leftarrow - \rightarrow$
& \!\!12\!\! & -3 & 4 & -8 & \!\!22\!\! 
& \!\!-21\!\! & 3 & 0 & \!\!-36\!\! & 9 & \!\!18\!\!
\\ \hline
\rule{0pt}{2.8mm}8 & $P_2H, P_i-P_j$ 
& $+ \rightarrow + \leftarrow - \leftarrow - \rightarrow$
& 6 & 0 & 1 & -2 & \!\!16\!\! 
& \!\!-12\!\! & 0 & 0 & \!\!-30\!\! 
& 6 & \!\!15\!\! \\ \hline    
\rule{0pt}{2.8mm}11 & $P_1,P_2,H$ 
& $+ \leftarrow + \rightarrow + \leftarrow - \rightarrow - \leftarrow - \rightarrow$ 
& 8 & -8 & 1 & -2 & 9 
& -8 & 8 & 0 & \!\!-16\!\!
& 0 & 8
\\ \hline
\rule{0pt}{2.8mm}12 & $P_1, H, H$ 
& $+ \leftarrow + \rightarrow + \leftarrow - \rightarrow - \leftarrow - \rightarrow$ 
& 1 & -1 & 0 & -1 & 2 
& -1 & 1 & 1 & -3
& 0 & 1 \\ \hline
\rule{0pt}{2.8mm}14 & $P_1, H, P_2H$ 
& $+ \leftarrow + \rightarrow + \leftarrow - \rightarrow - \leftarrow - \rightarrow$ 
& \!\!35\!\! & 0 & 4 & -8 & \!\!81\!\! 
& \!\!-70\!\! & 0 & 0 & \!\!\!\!-154\!\!\!\!
& \!\!35\!\! & \!\!77\!\! \\ \hline
\rule{0pt}{2.8mm}15 & \!\!$P_1, P_2H, P_2H$\!\!
& $+ \leftarrow + \rightarrow + \leftarrow - \rightarrow - \leftarrow - \rightarrow$ 
& \!\!85\!\! & 0 & 4 & -8 & \!\!20\!\! 
& \!\!\!\!-170\!\!\!\! & 0 & 0 & \!\!-32\!\! 
& \!\!85\!\! & \!\!16\!\! 
\\ \hline
\rule{0pt}{2.8mm}16 & \!\!$P_1, P_2H, H$\!\! 
& $+ \leftarrow + \rightarrow + \leftarrow - \rightarrow - \leftarrow - \rightarrow$ 
& \!\!62\!\! & 0 & 2 & -4 & \!\!64\!\! 
& \!\!\!\!-124\!\!\!\! & 0 & 0 & \!\!\!\!-124\!\!\!\!
& \!\!62\!\! & \!\!62\!\! \\ \hline
\end{tabular}
\caption{{\it Equivalence classes of chemical systems for equilibrium at $[1,1]$ with parameters of example chemical systems in each class, which have $8$-$10$ chemical reactions.}}
\label{table1}
\end{table}

\noindent
Two-dimensional homogeneous quadratic systems have sixteen geometric equivalence classes~\cite{Lyagina:1951:ICE,Newton:1978:TDH}. However, these results were obtained by considering arbitrary real parameters $a_{2,0}$, $a_{1,1}$, $a_{0,2}$, $b_{2,0}$, $b_{1,1}$ and $b_{0,2}$. In our case, Lemma~\ref{lemma1} restricts the choices of these coefficients by requiring inequalities~(\ref{ineq1})--(\ref{ineq2}). In this paper, we show that only nine of the sixteen geometric equivalence classes identified by Lyagina \cite{Lyagina:1951:ICE} can describe the dynamics of chemical reaction networks. The results are summarized in Table~\ref{table1}, where we use both Lyagina’s numbering~\cite{Lyagina:1951:ICE} and the notation introduced by Newton~\cite{Newton:1978:TDH}. The table also presents the results in terms of the notation used by Marcus~\cite{Markus:1960:QDE}, who classified the corresponding homogeneous quadratic systems into six classes. This classification was later expanded to seven classes by Newton~\cite{Newton:1978:TDH}, who observed that Marcus had overlooked one case. In terms of Marcus’s classification, Table~\ref{table1} shows that three classes can be realized by chemical systems. 

In general, these chemical systems have 12 nonzero parameters, corresponding to six second-order reactions ($a_{2,0}$, $a_{1,1}$, $a_{0,2}$, $b_{2,0}$, $b_{1,1}$, $b_{0,2}$), four first-order reactions ($a_{1,0}$, $a_{0,1}$, $b_{1,0}$, $b_{0,1}$) and two zero-order reactions ($a_{0,0}$, and $b_{0,0}$), giving chemical systems with $m=12$ chemical reactions. However, some of these parameters can be chosen to be zero. Examples in each geometric equivalence class can be constructed with $a_{0,2}=0$. The remaining coefficients are given in Table~\ref{table1}, which is further discussed in Section~\ref{sec36} . As shown in the table, for some geometric equivalence classes, examples can be found with as many as four zero coefficients, corresponding to chemical systems with only $m=8$ chemical reactions. However, we note that Table~\ref{table1} was obtained by imposing the restriction $a_{0,2}=0$ and conducting a computer search for suitable parameters satisfying the conditions of each geometric equivalence class. In some classes, chemical systems with fewer reactions do exist, as we discuss in Section~\ref{sec4}.

\subsection{Invariant lines}

To establish our results, summarized in Table~\ref{table1}, we will first study the invariant lines of the chemical system~(\ref{quadraticchemical1})--(\ref{quadraticchemical2}) which pass through the equilibrium point $[1,1]$. Such lines are given in the form
\begin{equation}
y \,=\, 1 \, + \, \eta \, (x-1) \,,
\label{invlines}
\end{equation}
where $\eta \in {\mathbb R} \cup \{\infty\}.$ If $\eta=\infty,$ then $x=1$, while the case $y=1$ corresponds to $\eta=0$ in equation~(\ref{invlines}). We formulate the criterion for finding invariant lines in our next lemma.

\vskip 2mm

\begin{lemma}
The chemical system~$(\ref{quadraticchemical1})$--$(\ref{quadraticchemical2})$ has invariant lines through the equilibrium point $[1,1]$ of the form~$(\ref{invlines})$
provided that $\eta$ is the solution of the equation
\begin{equation}
\sigma_3 \, \eta^3
\,-\,
\sigma_2 \, \eta^2
\,+\,
\sigma_1 \, \eta
\,-\,
\sigma_0
\,=\,
0
\label{cubicinvlinesigma}    
\end{equation}
where
\begin{equation}
\sigma_3 = a_{0,2}\,,
\qquad
\sigma_2 = b_{0,2}-a_{1,1}\,,
\qquad
\sigma_1 = a_{2,0}-b_{1,1}
\qquad\mbox{and}\qquad
\sigma_0 = b_{2,0}\,. 
\label{sigmavalues}
\end{equation}
Moreover, the coefficients $\sigma_i\,,$ $i=1,2,3,4$, satisfy the inequalities
\begin{equation}
\sigma_1 \ge \sigma_3 \ge 0\,
\qquad\mbox{and}\qquad
\sigma_2 \ge \sigma_0 \ge 0\,.
\label{sigmaineq}
\end{equation} 
In particular, we have one, two or three invariant lines~$(\ref{invlines})$ and they are all non-decreasing, i.e. all solutions of $(\ref{cubicinvlinesigma})$ satisfy the non-negativity condition
\begin{equation}
\eta \ge 0\,.
\label{nonnegcond}    
\end{equation}
\label{lemma2}
\end{lemma}
\begin{proof}
Introducing new variables~(\ref{uvvariables}),
the planar ODE system (\ref{quadraticchemical1})--(\ref{quadraticchemical2})
transforms to the ODE system~(\ref{homODE1})--(\ref{homODE2}). Dividing equation~(\ref{homODE2}) by equation~(\ref{homODE1}), we obtain
\begin{equation}
\frac{\mbox{{\rm d}}v}
{\mbox{{\rm d}}u}
=
\frac{
b_{2,0} \, u^2
\,+\,
b_{1,1} \, u \, v
\,+\,
b_{0,2} \, v^2}
{a_{2,0} \, u^2
\,+\,
a_{1,1} \, u \, v
\,+\,
a_{0,2} \, v^2}
\,.
\label{homODE12}
\end{equation}
The invariant lines~(\ref{invlines}) in the new variables take the form $v\,=\,\eta\,u$. Substituting into equation~(\ref{homODE12}), we obtain that $v\,=\,\eta\,u$ is an invariant line if and only if
$$
\eta
=
\frac{
b_{2,0} 
\,+\,
b_{1,1} \, \eta
\,+\,
b_{0,2} \, \eta^2}
{a_{2,0} 
\,+\,
a_{1,1} \, \eta
\,+\,
a_{0,2} \, \eta^2}\,,
$$
which can be rewritten as
\begin{equation}
a_{0,2} \, \eta^3
\,+\,
\big(a_{1,1}-b_{0,2}\big) \, \eta^2
\,+\,
\big(a_{2,0}-b_{1,1}\big) \, \eta
\,-\,
b_{2,0} 
\,=\,
0\,.
\label{cubicinvline}    
\end{equation}
Substituting formulas~(\ref{sigmavalues}) for coefficients $\sigma_i\,,$ $i=1,2,3,4$, into equation~(\ref{cubicinvline}), we obtain~(\ref{cubicinvlinesigma}). Using inequalities~(\ref{ineq1})--(\ref{ineq2}), we obtain
\begin{eqnarray*}
\sigma_3 &=& a_{0,2} \,\ge\, 0 \,,
\\
\sigma_2 &=& b_{0,2} - a_{1,1}
\,\ge\, - b_{2,0} - b_{1,1}\,\ge\,b_{2,0}\,=\,\sigma_0
\,,
\\
\sigma_1 &=& a_{2,0}-b_{1,1} 
\,\ge\,-a_{1,1}-a_{0,2}\,\ge\,a_{0,2}\,=\,\sigma_3\,,
\\
\sigma_0 &=& b_{2,0}\,\ge\,0\,,
\end{eqnarray*}
which establishes inequalities~(\ref{sigmaineq}). The non-negativity condition~(\ref{nonnegcond}) follows from inequalities~(\ref{sigmaineq}) by applying Descartes' rule of signs.
\end{proof}

\subsection{Chemical triangle}

\label{sec32}

Next, we adapt the approach of Newton~\cite{Newton:1978:TDH}, who introduces a two-parameter family of systems, together with the so-called zero-derivative and undefined-derivative lines, to classify general ODE systems. Since chemical systems impose additional sign restrictions on the coefficients, we introduce a different pair of parameters from those used by Newton to simplify the subsequent analysis. We define nonnegative parameters $\lambda$ and $\mu$ by
\begin{equation}
\lambda \,=\, - \,a_{1,1}
\qquad
\mbox{and}
\qquad
\mu \,=\, - \,b_{1,1}\,.
\label{defmunu}
\end{equation}
Using~(\ref{sigmavalues}) and notation~(\ref{defmunu}), the ODE system~(\ref{homODE1})--(\ref{homODE2}) can be rewritten as
\begin{eqnarray}
\frac{\mbox{{\rm d}}u}{\mbox{{\rm d}}t}
\,& = &\,
(\sigma_1-\mu) \, u^2
\,-\,
\lambda \, u \, v
\,+\,
\sigma_3 \, v^2 \,,
\label{homODE1lambdamu}
\\
\frac{\mbox{{\rm d}}v}{\mbox{{\rm d}}t}
\,& = &\,
\sigma_0 \, u^2
\,-\,
\mu \, u \, v
\,+\,
(\sigma_2 - \lambda) \, v^2 \,.
\label{homODE2lambdamu}
\end{eqnarray}

\vskip 2mm

\begin{lemma}
Consider nonnegative parameters $\lambda$, $\mu$, $\sigma_i\,,$ $i=1,2,3,4$, satisfying inequalities~$(\ref{sigmaineq})$. Then the planar ODE system~$(\ref{homODE1lambdamu})$--$(\ref{homODE2lambdamu})$ is a representation of the chemical system~$(\ref{quadraticchemical1})$--$(\ref{quadraticchemical2})$ provided that 
\begin{equation}
\lambda \ge 2 \sigma_3\,,
\qquad
\mu \ge 2 \sigma_0
\qquad
\mbox{and}
\qquad
\lambda+\mu \le \min\big\{\sigma_0 + \sigma_2,\,
\sigma_1 + \sigma_3\big\}\,.
\label{chemicaltriangle}
\end{equation}
If $\sigma_3 \ne 0$, then the invariant lines through the equilibrium point $[1,1]$ of the form~$(\ref{invlines})$ are given by $\eta$
solving the cubic equation $(\ref{cubicinvlinesigma})$. Let $\eta_i$ be a solution of the cubic equation $(\ref{cubicinvlinesigma})$. Consider function
\begin{equation}
f(\eta)
=
\frac{
\sigma_0
\,-\,
\mu \, \eta
\,+\,
(\sigma_2-\lambda) \, \eta^2}
{\sigma_1 - \mu
\,-\,
\lambda \, \eta
\,+\,
\sigma_3 \, \eta^2}\,.
\label{deffeta}
\end{equation} 
Then $f^\prime(\eta_i)=0$ on the line $Z[\eta_i]$ in $(\lambda,\mu)$-parameter space given by
\begin{equation}
\eta_i \lambda + \mu 
\,=\, 2 \eta_i 
(\sigma_2 - \sigma_3 \eta_i)\,,
\label{zeroderline}
\end{equation}
and $f^\prime(\eta_i)$ is undefined on the line $U[\eta_i]$ in $(\lambda,\mu)$-parameter space given by
\begin{equation}
\eta_i \lambda + \mu 
\,=\, \sigma_1 + \sigma_3 \eta_i^2.
\label{undefline}
\end{equation}
\label{lemma3}
\end{lemma}

\begin{proof}
Using inequalities~(\ref{ineq1})--(\ref{ineq2}), we obtain inequalities~(\ref{chemicaltriangle}). Since $\eta_i$ is a solution of the cubic equation~(\ref{cubicinvlinesigma}), we have $f(\eta_i)=\eta_i$. Differentiating~(\ref{deffeta}) and substituting $f(\eta_i)=\eta_i$, we obtain
$$
f^\prime(\eta_i)
=
\frac{
- \, \eta_i \lambda \, - \,\mu
\,+\,
2 \sigma_2 \, \eta_i
-
2 \sigma_3 \, \eta^2_i
}
{
- \, \eta_i \lambda \, - \,\mu
\,+\,
\sigma_1 
\,+\,
\sigma_3 \, \eta_i^2}\,,
\quad
\mbox{which implies~(\ref{zeroderline}) and~(\ref{undefline})}.
$$
\vskip -7mm
\end{proof}

\vskip 2mm

\noindent
The inequalities~(\ref{chemicaltriangle}) describe a triangle in the $(\lambda,\mu)$-parameter plane. We will call it the {\it chemical triangle} and denote it by
\begin{equation}
{\mathcal T} (\sigma_0,\sigma_1,\sigma_2,\sigma_3)
\,=\,
\big\{
(\lambda,\mu) \in {\mathbb R}^2 \, \big|
\,
\lambda \; \mbox{and} \; \mu
\; \mbox{satisfy inequatilities~(\ref{chemicaltriangle})}
\big\}\,.
\label{chemicaltriangle2}
\end{equation}
The classification of the general ODE systems into sixteen geometric equivalence classes~\cite{Newton:1978:TDH} depend on the relative positions of the zero-derivative and undefined-derivative lines $Z[\eta_i]$ and $U[\eta_i]$, which we have defined in Lemma~\ref{lemma3} for $\sigma_3 \ne 0$. In particular, considering chemical systems, we need to establish the relative positions of these lines and the chemical triangle~${\mathcal T} (\sigma_0,\sigma_1,\sigma_2,\sigma_3)$. Considering a line given by $s_1 \lambda + s_2 \mu = s_3$ for $s_1, s_2, s_3 \in \mathbb{R}$, we say that the chemical triangle~${\mathcal T} (\sigma_0,\sigma_1,\sigma_2,\sigma_3)$ is {\it bounded above} by this line if all the points in the chemical triangle~${\mathcal T} (\sigma_0,\sigma_1,\sigma_2,\sigma_3)$ satisfy $s_1\lambda + s_2\mu \le s_3$. Similarly, we say that the chemical triangle~${\mathcal T} (\sigma_0,\sigma_1,\sigma_2,\sigma_3)$ is {\it bounded below} by the line $s_1\lambda + s_1\mu = s_3$ if all the points in the chemical triangle~${\mathcal T} (\sigma_0,\sigma_1,\sigma_2,\sigma_3)$ satisfy $s_1\lambda + s_2\mu \ge s_3$.

Following Newton~\cite{Newton:1978:TDH}, we will extend the definition of $Z[\eta_i]$ and $U[\eta_i]$ to the case $\sigma_3 = 0$, when one of the invariant lines is given by $x=1$, corresponding to $\eta=\infty$ in equation~(\ref{invlines}). We define 
\begin{equation}
Z[\infty] \; \mbox{as line} \; \lambda = 0
\quad \mbox{and} \quad
U[\infty] \; \mbox{as line} \; \lambda = \sigma_2.
\label{infinitycase}    
\end{equation}
We observe that the undefined-derivative line $U[\infty]$ does not intersect the interior of the chemical triangle~${\mathcal T} (\sigma_0,\sigma_1,\sigma_2,\sigma_3)$. A similar result can be proved for all values of $\eta_i$.

\vskip 2mm

\begin{lemma}
Let $\eta_i$ be a solution of the equation $(\ref{cubicinvlinesigma})$. Then the interior of the chemical triangle~${\mathcal T} (\sigma_0,\sigma_1,\sigma_2,\sigma_3)$ defined by~$(\ref{chemicaltriangle2})$ is bounded above by the undefined-derivative line $U[\eta_i]$. 
\label{lemma4}
\end{lemma}

\begin{proof}
We consider two cases: $\eta_i \le 1$ and $\eta_i > 1$.
For $\eta_i \le 1$, we deduce using~(\ref{undefline}) that any parameter pair $(\lambda,\mu)$ on the undefined-derivative line $U[\eta_i]$ satisfies
$$
\lambda+\mu
\,=\,
\eta_i \lambda + \mu
+ (1-\eta_i) \lambda
\,=\, \sigma_1 + \sigma_3 \eta_i^2
+ (1-\eta_i) \lambda
\,=\, \sigma_1 + \sigma_3 +
\sigma_3 (\eta_i^2-1)
+ (1-\eta_i) \lambda.
$$
Consequently, if $\lambda \ge 2 \sigma_3$, then since $1-\eta_i\ge0$, we have
$$
\lambda+\mu \,\ge\,  \sigma_1 + \sigma_3 +
\sigma_3 (\eta_i^2-1)
+ (1-\eta_i) 2 \sigma_3
\,=\,\sigma_1 + \sigma_3 +
\sigma_3 (\eta_i-1)^2
\,\ge\,\sigma_1 + \sigma_3\,\ge\,\min\big\{\sigma_0 + \sigma_2,\,
\sigma_1 + \sigma_3\big\}\,.
$$
For $\eta_i > 1$, note that~(\ref{undefline}) can be rewritten as
$$
\eta_i^2 \lambda + \eta_i\mu 
\,=\, \sigma_0 + \sigma_2 \eta_i^2.
$$
Thus, we deduce using~(\ref{undefline}) that any parameter pair $(\lambda,\mu)$ on the undefined-derivative line $U[\eta_i]$ satisfies
$$
\lambda+\mu
\,=\,
\frac{1}{\eta_i^2}\big(\eta_i^2 \lambda + \eta_i \mu
+ (\eta_i^2-\eta_i) \mu\big)
\,=\, \frac{1}{\eta_i^2}\big(\sigma_0 + \sigma_2 \eta_i^2
+ (\eta_i^2-\eta_i) \mu\big)
\,=\, 
\sigma_2
+
\frac{1}{\eta_i^2}\big(\sigma_0
+ (\eta_i^2-\eta_i) \mu\big).
$$
Consequently, if $\mu \ge 2 \sigma_0$, then since $\eta_i^2-\eta_i>0$, we have
$$
\lambda+\mu \,\ge\,  \sigma_2 
+
\frac{1}{\eta_i^2}\big(\sigma_0
+ (\eta_i^2-\eta_i)2\sigma_0\big)
\,=\,\sigma_0 + \sigma_2 + \frac{1}{\eta_i^2}\sigma_0(\eta_i-1)^2
\,\ge\,\sigma_0 + \sigma_2\,.
$$
In both cases, we have $\lambda + \mu \,\ge\,\min\big\{\sigma_0 + \sigma_2,\,
\sigma_1 + \sigma_3\big\}\,$, and thus, we conclude that all the points on the line $U[\eta_i]$ lie above or on the line $\lambda+\mu = \min\big\{\sigma_0 + \sigma_2,\,
\sigma_1 + \sigma_3\big\}\,$. Therefore, the chemical triangle~(\ref{chemicaltriangle2}) is bounded above by $U[\eta_i]$.
\end{proof}

\vskip 2mm

\begin{example}
\label{example5}
Consider the following system of $m=10$ chemical reactions:
\begin{equation}
\emptyset
\;
\mathop{\longrightarrow}^{1}
\;
X+Y \, ,
\;\quad
X
\;
\mathop{\longrightarrow}^{2}
\;
\emptyset \, ,
\;\quad
Y
\;
\mathop{\longrightarrow}^{2}
\;
\emptyset \, ,
\;\quad
X
\;
\mathop{\longrightarrow}^{4}
\;
Y\, ,
\;\quad
Y
\;
\mathop{\longrightarrow}^{4}
\;
X\, ,
\;\quad
X \, + \, Y 
\;
\mathop{\longrightarrow}^{12}
\;
\emptyset \,,
\label{crnex5a}
\end{equation}
\begin{equation}
2 X
\;
\mathop{\longrightarrow}^{5}
\;
3X \, ,
\qquad
2 Y
\;
\mathop{\longrightarrow}^{5}
\;
3 Y \, ,
\qquad
2 X
\;
\mathop{\longrightarrow}^{4}
\;
3 X + Y
\quad
\mbox{and}
\quad
2 Y
\;
\mathop{\longrightarrow}^{4}
\;
X + 3 Y \, .
\label{crnex5b}    
\end{equation}
Then the ODEs~(\ref{quadraticchemical1})--(\ref{quadraticchemical2}) become
\begin{eqnarray}
\frac{\mbox{{\rm d}}x}{\mbox{{\rm d}}t}
\,& = &\,
1
\,-\,
6 \, x
\,+\,
4 \, y
\,+\,
9 \, x^2
\,-\,
12 \, x \, y
\,+\,
4 \, y^2 \,,
\label{quadraticchemical1ex5}
\\
\frac{\mbox{{\rm d}}y}{\mbox{{\rm d}}t}
\,& = &\,
1
\,+\,
4 \, x
\,-\,
6 \, y
\,+\,
4 \, x^2
\,-\,
12 \, x \, y
\,+\,
9 \, y^2 \,,
\label{quadraticchemical2ex5}
\end{eqnarray}
which satisfy relations~(\ref{othcoef1})--(\ref{othcoef6}). Using~(\ref{sigmavalues}), the coefficients $\sigma_i$, $i=0,1,2,3,$ are
$$
\sigma_3 = 4\,,
\qquad
\sigma_2 = 21\,,
\qquad
\sigma_1 = 21\,
\qquad\mbox{and}\qquad
\sigma_0 = 4\,. 
$$
The solutions of cubic polynomial~(\ref{cubicinvlinesigma}) are
\begin{equation}
\eta_1 = \frac{1}{4}\,,
\qquad
\eta_2 = 1
\quad
\mbox{and}
\quad
\eta_3 = 4.
\label{etavalues}
\end{equation}
The phase portrait is plotted in Figure~\ref{figure2}(a), where we plot the three invariant lines going through the equilibrium point [1,1] as black dashed lines. They divide the non-negative quadrant $\overline{\Omega}$ into six sectors: two hyperbolic sectors and four parabolic sectors. The ODE system~(\ref{quadraticchemical1ex5})--(\ref{quadraticchemical2ex5}) can be also rewritten in the form~(\ref{homODE1lambdamu})--(\ref{homODE2lambdamu}), where
$
\lambda = \mu = 12.
$
These values satisfy inequalities~(\ref{chemicaltriangle}), which describe the chemical triangle~~${\mathcal T} (\sigma_0,\sigma_1,\sigma_2,\sigma_3)$ in the $(\lambda,\mu)$-parameter plane. This triangle is visualized as the red triangle in Figure~\ref{figure2}(b) together with the lines $Z[\eta_i]$ and $U[\eta_i]$ defined by~(\ref{zeroderline}) and~(\ref{undefline}). 
\end{example}

\begin{figure}
\rule{0pt}{1pt}
\vskip 1mm
\rule{0pt}{1pt}
\hskip -1.5cm
\epsfig{file=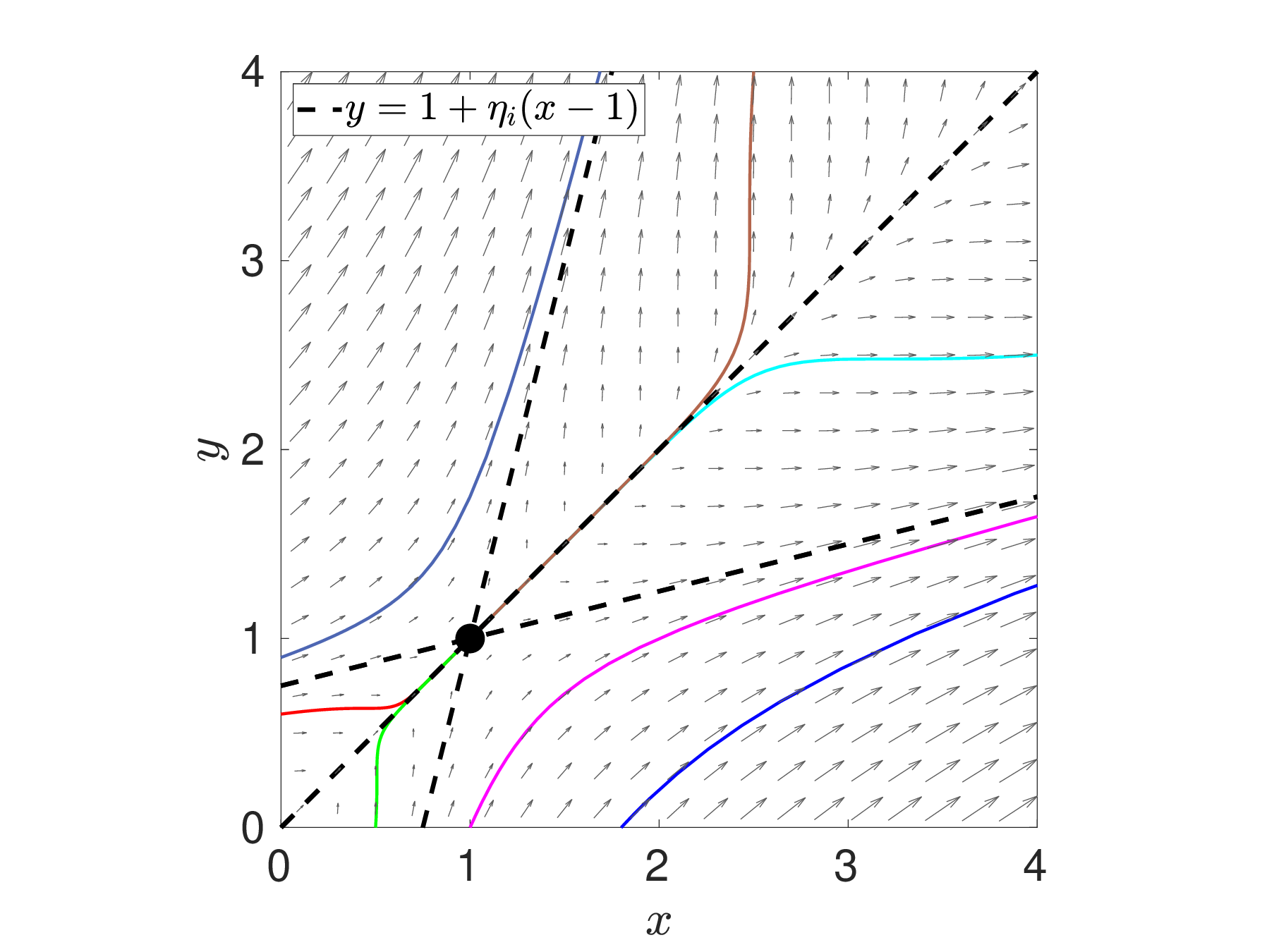,height=7.4cm}
\hskip -1.5cm
\epsfig{file=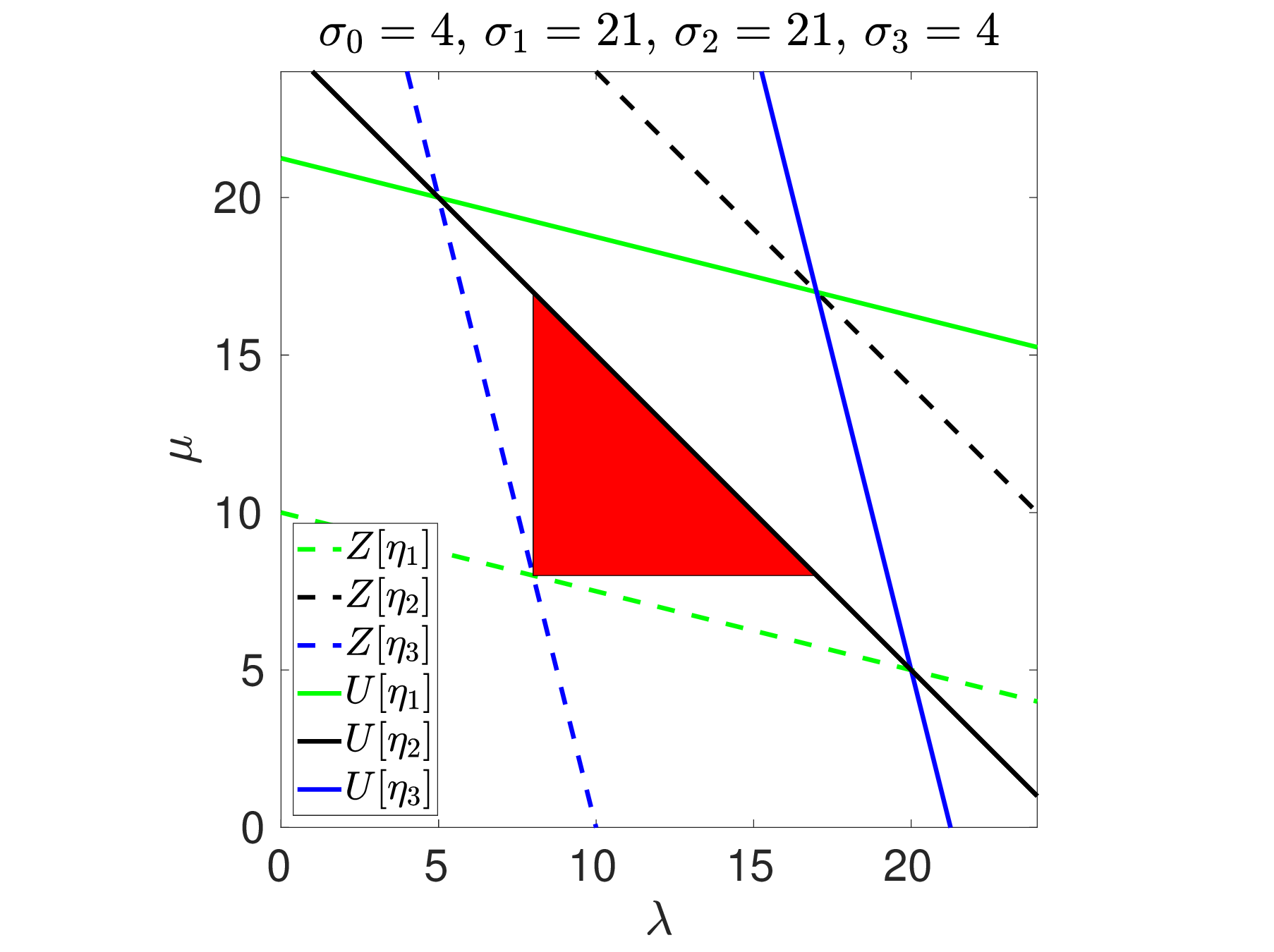,height=7.4cm}
\vskip -7.6cm
(a) \hskip 76mm (b)
\vskip 7.1cm
\caption{{\it {\rm (a)} The phase plane of the ODE system~$(\ref{quadraticchemical1ex5})$--$(\ref{quadraticchemical2ex5})$ describing the chemical reaction network~$(\ref{crnex5a})$--$(\ref{crnex5b})$. We plot seven illustrative trajectories in the non-negative quadrant $\overline{\Omega}$ together with the invariant lines $y=1 + \eta_i(x-1),$ where their slopes $\eta_i$, $i=1,2,3,$ are given by~$(\ref{etavalues})$. 
\hfill\break {\rm (b)} The chemical triangle~$(\ref{chemicaltriangle2})$ in the $(\lambda,\mu)$-parameter space is visualized as the red triangle together with the zero-derivative lines $Z[\eta_i]$ given by~$(\ref{zeroderline})$ and undefined-derivative lines $U[\eta_i]$ given by~$(\ref{undefline})$. 
}}
\label{figure2}
\end{figure}

\vskip 3mm

\noindent
In Figure~\ref{figure2}(b), we observe that the lines $U[\eta_i]$ and $Z[\eta_i]$ do not intersect with the interior of the chemical triangle ${\mathcal T} (\sigma_0,\sigma_1,\sigma_2,\sigma_3)$. While this result has been established for lines $U[\eta_i]$ in its full generality in Lemma~\ref{lemma3}, there exist chemical systems which have one line $Z[\eta_i]$ intersecting with the interior of the chemical triangle ${\mathcal T} (\sigma_0,\sigma_1,\sigma_2,\sigma_3)$. In order to study the relative positions of the zero-derivative lines $Z[\eta_i]$ and the chemical triangle~${\mathcal T} (\sigma_0,\sigma_1,\sigma_2,\sigma_3)$, we first study the case when the equation~$(\ref{cubicinvlinesigma})$ is a cubic equation with three distinct real roots.

\vskip 2mm

\begin{lemma}
Let $\sigma_3 \ne 0$ and consider that the cubic equation~$(\ref{cubicinvlinesigma})$ has three real solutions
$$
\eta_1 < \eta_2 < \eta_3.
$$
Then the chemical triangle~$(\ref{chemicaltriangle2})$ is non-empty, i.e. ${\mathcal T} (\sigma_0,\sigma_1,\sigma_2,\sigma_3) \ne \emptyset$, if and only if one of the following three conditions holds \vskip 2mm
\noindent \rule{0pt}{1pt} \hskip 1cm {\rm (i)} $\;\;\; \eta_2 = 1\,$, \vskip 1mm
\noindent \rule{0pt}{1pt} \hskip 0.95cm {\rm (ii)} $\;\;\, \eta_2>1$ and
\begin{equation}
\eta_1 \le \frac{\eta_2 + \eta_3 -2}{\eta_2 \eta_3 -1}\,,
\label{restrictioneta2g1}
\end{equation}
\rule{0pt}{1pt} \hskip 0.9cm {\rm (iii)} $\;\;\eta_2 < 1$ and 
\begin{equation}
\eta_3 \ge \frac{1 - \eta_1 \eta_2}{\eta_1 + \eta_2 - 2 \eta_1 \eta_2}\,.
\label{restrictioneta2l1}
\end{equation}
\vskip 1mm
\noindent
In each case, the chemical triangle~${\mathcal T} (\sigma_0,\sigma_1,\sigma_2,\sigma_3)$ defined by~$(\ref{chemicaltriangle2})$ lies on the side of $U[\eta_2]$ opposite the line $Z[\eta_2]$ and the same side of $U[\eta_i]$ as the line $Z[\eta_i]$ for $i = 1,3$. Moreover, the interior of the chemical triangle~${\mathcal T} (\sigma_0,\sigma_1,\sigma_2,\sigma_3)$ cannot by intersected by both lines $Z[\eta_1]$ and $Z[\eta_3]$ at the same time.
\label{lemma5}
\end{lemma}

\begin{proof}
Using~(\ref{chemicaltriangle}), we obtain that the chemical triangle~(\ref{chemicaltriangle2}) is non-empty if and only if
$$
2 \sigma_3 + 2 \sigma_0
\,\le\,
\min\big\{\sigma_0 + \sigma_2,\,
\sigma_1 + \sigma_3\big\}\,,
$$
which can be rewritten as two inequalities
\begin{equation}
2 \sigma_3 + \sigma_0
\,\le\,
\sigma_2\,,
\qquad
\mbox{and}
\qquad
\sigma_3 + 2 \sigma_0
\,\le\,
\sigma_1\,.
\label{auxineq}
\end{equation}
Since $\eta_i$, $i=1,2,3$, are roots of the cubic equation~(\ref{cubicinvlinesigma}), we have
\begin{equation}
\sigma_2 = \sigma_3 (\eta_1 + \eta_2 + \eta_3)\,,
\qquad
\sigma_1 = \sigma_3 (\eta_1\eta_2 + \eta_1 \eta_3 + \eta_2 \eta_3)\,,
\quad
\mbox{and}
\quad
\sigma_0 = \sigma_3 \eta_1 \eta_2 \eta_3\,.
\label{sig210expr}
\end{equation}
Substituting into~(\ref{auxineq}), we get two inequalities
\begin{equation}
2  + \eta_1 \eta_2 \eta_3
\,\le\,
\eta_1 + \eta_2 + \eta_3
\qquad
\mbox{and}
\qquad
1 + 2 \eta_1 \eta_2 \eta_3
\,\le\,
\eta_1\eta_2 + \eta_1 \eta_3 + \eta_2 \eta_3\,.
\label{auxineq2}
\end{equation}
If $\eta_2=1$, then both inequalities (\ref{auxineq2}) simplify to
$0 \le (1-\eta_1)(\eta_3-1)$,
which is valid for all $\eta_1<1$ and $\eta_3>1.$ If $\eta_2 > 1$, then inequalities (\ref{auxineq2}) can be rewritten as the restriction~(\ref{restrictioneta2g1}) on $\eta_1$ from above. If $\eta_2 < 1$, then inequalities (\ref{auxineq2}) can be rewritten as the restriction~(\ref{restrictioneta2l1}) on $\eta_3$ from below. 

Using~\cite[Proposition 5]{Newton:1978:TDH}, we know that the lines $U[\eta_1],$ $U[\eta_2]$ and $U[\eta_3]$ form a triangle of nonzero area, and line $Z[\eta_j]$ passes through the intersection of lines $U[\eta_i]$ and $U[\eta_k]$ as $(i,j,k)$ cycles through $(1,2,3)$ (this result is also illustrated in Figure~\ref{figure2}(b)). Let the coordinates of the intersection between $U[\eta_i]$ and $U[\eta_k]$ be $(\mu_{ik},\lambda_{ik})$. 
Using~(\ref{undefline}), we obtain 
$$
\lambda_{ik}\,=\,\sigma_3(\eta_i+\eta_k)
\qquad
\mbox{and} 
\qquad
\mu_{ik}=\sigma_1 - \sigma_3 \eta_i \eta_k\,.
$$
Substituting $\mu_{ik}$ into $U[\eta_j]$, we get $\lambda = \sigma_3 (\eta_j^2 + \eta_i\eta_k)/\eta_j$
at $\mu = \mu_{ik}$. Subtracting this $\lambda$ from $\lambda_{ik}$ gives $\sigma_3 (\eta_j-\eta_i)(\eta_k-\eta_j)/\eta_j$, which is negative for $j = 1$ and $j=3$, while it is positive for $j = 2$. Therefore, the intersection of lines $U[\eta_2]$ and $U[\eta_3]$ lies below 
$U[\eta_1]$, the intersection of lines $U[\eta_1]$ and $U[\eta_2]$ lies below $U[\eta_3]$, and the intersection of lines $U[\eta_1]$ and $U[\eta_3]$ lies above $U[\eta_2]$. Since the chemical triangle~${\mathcal T} (\sigma_0,\sigma_1,\sigma_2,\sigma_3)$ defined by~$(\ref{chemicaltriangle2})$ lies below $U[\eta_2]$, it is on the side of $U[\eta_2]$ opposite the line $Z[\eta_2]$ and on the same sides of $U[\eta_1]$ and $U[\eta_3]$ as $Z[\eta_1]$ and $Z[\eta_3]$, respectively.

Finally, let us assume that both lines $Z[\eta_1]$ and $Z[\eta_3]$ intersect the chemical triangle~${\mathcal T} (\sigma_0,\sigma_1,\sigma_2,\sigma_3)$ in interior points. Then the line $Z[\eta_1]$ must intersect the line $\lambda=2\sigma_3$ at point $\mu > 2 \sigma_0$ and the line $Z[\eta_3]$ must intersect the line $\mu=2 \sigma_0$ at point $\lambda > 2 \sigma_3.$ Substituting into~(\ref{zeroderline}), we obtain two inequalities 
$$
\sigma_0 < \eta_1 (\sigma_2 - \sigma_3 - \sigma_3 \eta_1)
\qquad
\mbox{and}
\qquad
\sigma_0 < \eta_3 (\sigma_2 - \sigma_3 - \sigma_3 \eta_3) \,.
$$
Substituting the expressions for $\sigma_0$ and $\sigma_2$ given by~(\ref{sig210expr}), these inequalities can be rewritten as
$$
0 < (\eta_3 - 1) (1 - \eta_2)
\qquad
\mbox{and} 
\qquad
0 < (1 - \eta_1) (\eta_2 - 1)\,,
$$
which cannot be satisfied at the same time, because $\eta_3>1$ and $\eta_1<1$ holds in all cases (i)--(iii) when the chemical triangle~${\mathcal T} (\sigma_0,\sigma_1,\sigma_2,\sigma_3)$ is non-empty. Thus we proved by contradiction that the interior of the chemical triangle~${\mathcal T} (\sigma_0,\sigma_1,\sigma_2,\sigma_3)$ cannot by intersected by both lines $Z[\eta_1]$ and $Z[\eta_3]$.
\end{proof}

\subsection{Geometric equivalence classes with three invariant lines}

\label{sec33}

Newton~\cite{Newton:1978:TDH} follows classification of Shilov~\cite{Shilov:1950:ICH} of invariant lines as parabolic of first kind (denoted $P_1$), parabolic of second kind ($P_2$) and hyperbolic ($H$). For example, ODEs~(\ref{quadraticchemical1ex5})--(\ref{quadraticchemical2ex5}) in Example~\ref{example5} are in Newton's class $H(\eta_1)$, $P_1(\eta_2)$, $H(\eta_3)$ which can be deduced directly from the position of point $(\lambda,\mu)=(12,12)$ in Figure~\ref{figure2}(b). We note that Newton would also write this class as $(P_1,H,H)$ because any permutation of indices leads to the same geometric equivalence class. This class is numbered 12 in Lyagina's classification~\cite{Lyagina:1951:ICE}.

In general, using \cite[Proposition 4]{Newton:1978:TDH}, we can use Lemma~\ref{lemma5} to limit the possible classifications of invariant lines for $(\mu,\lambda)$ in the chemical triangle. We observe that, since the chemical triangle~$(\ref{chemicaltriangle2})$ lies on the side of $U[\eta_2]$ opposite to $Z[\eta_2]$, the invariant line corresponding to $\eta=\eta_2$ can only be classified as $P_1(\eta_2)$. We also note that the fact that the chemical triangle is on the same sides of $U[\eta_1]$ and $U[\eta_3]$ as $Z[\eta_1]$ and $Z[\eta_3]$ respectfully means that the invariant lines corresponding to $\eta=\eta_1$ and $\eta=\eta_3$ cannot be categorised as $P_1(\eta_i)$. Moreover, since the interior of the chemical triangle~${\mathcal T} (\sigma_0,\sigma_1,\sigma_2,\sigma_3)$ cannot be intersected by both lines $Z[\eta_1]$ and $Z[\eta_3]$ at the same time, we cannot have both of the invariant lines corresponding to $\eta=\eta_1$ and $\eta=\eta_3$ classified as $P_2(\eta_i)$. Therefore, for $\sigma_3\,\ne\, 0$ and three distinct real roots, the only possible combinations of classification of invariant lines are $(P_1,P_2,H),\, (P_1,H,H),\,(P_1,P_2,P_2H),\,(P_1,H,P_2H)$ and $(P_1, P_2H, P_2H)$, or in terms of Lyagina's geometric equivalence classes~\cite{Lyagina:1951:ICE}, we can realize classes 11, 12, 14, 15 and 16 as chemical system. The same conclusion can also be derived for the limiting case $\sigma_3=0$, where the equation~$(\ref{cubicinvlinesigma})$ reduces to the quadratic equation (with $\sigma_2 \ne 0$)
\begin{equation}
\,-\,
\sigma_2 \, \eta^2
\,+\,
\sigma_1 \, \eta
\,-\,
\sigma_0
\,=\,
0\,
\label{reducedcubicinvlinesigma}    
\end{equation}
because we can apply the following lemma.

\vskip 2mm

\begin{lemma}
Let $\sigma_3 = 0$  and consider that the quadratic equation~$(\ref{reducedcubicinvlinesigma})$ has two nonnegative solutions
$$
\eta_1 < \eta_2.
$$
Then the chemical triangle~$(\ref{chemicaltriangle2})$ is non-empty, i.e. ${\mathcal T} (\sigma_0,\sigma_1,\sigma_2,0) \ne \emptyset$, if and only if one of the following two conditions holds \vskip 2mm
\noindent \rule{0pt}{1pt} \hskip 1cm {\rm (i)} $\;\;\; \eta_2  \le 1\,$, \vskip 1mm
\noindent \rule{0pt}{1pt} \hskip 0.95cm {\rm (ii)} $\;\;\, \eta_2>1$ and $\eta_1 < 1/\eta_2$\,.
\rule{0pt}{1pt} \hskip 0.9cm
\vskip 2mm
\noindent
In each case, the line $Z[\eta_2]$ does not intersect the chemical triangle~${\mathcal T} (\sigma_0,\sigma_1,\sigma_2,0)$ defined by~$(\ref{chemicaltriangle2})$. Moreover, $Z[\infty]$ does not intersect the interior of the chemical triangle~${\mathcal T} (\sigma_0,\sigma_1,\sigma_2,0)$.
\label{lemma6}
\end{lemma}
    
\vskip 2mm

\begin{proof}
Substituting $\sigma_3 = 0$ into equation~(\ref{auxineq}), we obtain
\begin{equation}
\sigma_0
\,\le\,
\sigma_2
\qquad
\mbox{and}
\qquad
2 \sigma_0
\,\le\,
\sigma_1\,.
\label{auxineql6}
\end{equation}
Since $\eta_1$ and $\eta_2$ are roots of the quadratic equation~(\ref{reducedcubicinvlinesigma}), we have
$$
\sigma_1 = \sigma_2 (\eta_1 + \eta_2)
\qquad
\mbox{and}
\qquad
\sigma_0 = \sigma_2 \eta_1 \eta_2\,.
$$
Substituting into~(\ref{auxineql6}), we get
\begin{equation}
\eta_1 \le \frac{1}{\eta_2}
\qquad
\mbox{and}
\qquad
\eta_1(2\eta_2 -1) \le \eta_2\,.
\label{auxineq2l6}
\end{equation}
If $\eta_2\,\le\,1$, then $2\eta_2-1\,\le\,1$, so $\eta_1(2\eta_2-1)\,\le\,\eta_1\,<\,\eta_2$, and thus~(\ref{auxineq2l6}) is satisfied for all $\eta_2\,\le \,1$. If $\eta_2\,>\,1$, then the first inequality in~(\ref{auxineq2l6}) implies the second, because
$$
\eta_1 
\,\le\,
\frac{1}{\eta_2}
\,\le\,
\frac{\eta_2}{2\eta_2-1}\,.
$$
As in the proof of Lemma~\ref{lemma5},  we use~\cite[Proposition 5]{Newton:1978:TDH} to deduce that the lines $U[\eta_1],$ $U[\eta_2]$ and $U[\infty]$ form a triangle of nonzero area and line $Z[\eta_2]$ passes through the intersection of lines $U[\eta_1]$ and $U[\infty]$. However, in this case, $\lambda_{13} = \sigma_2$, and thus, substituting this into the equation for $U[\eta_1]$, we find that $\mu_{13} = \sigma_1 -\sigma_2\eta_1$. Substituting this into the equation for $U[\eta_2]$, we see that $\lambda = \sigma_2 \eta_1/\eta_2< \lambda_{13}$.Thus, like in our proof of Lemma~\ref{lemma3}, we conclude that the chemical triangle~${\mathcal T} (\sigma_0,\sigma_1,\sigma_2, 0)$ defined by~$(\ref{chemicaltriangle2})$ lies on the side of $U[\eta_2]$ opposite the line $Z[\eta_2]$. Finally, note that the horizontal side of the chemical triangle~${\mathcal T} (\sigma_0,\sigma_1,\sigma_2,0)$ lies on the line $\lambda = 0$, which, using~(\ref{infinitycase}), is also $Z[\infty]$. Therefore, $Z[\infty]$ does not intersect the interior of the chemical triangle~${\mathcal T} (\sigma_0,\sigma_1,\sigma_2,0)$. 
\end{proof}

\vskip 2mm

\noindent
Using \cite[Proposition 4]{Newton:1978:TDH} and Lemma~\ref{lemma6}, we can make the same conclusion for the case $\sigma_3=0$ as we have argued for $\sigma_3 \ne 0$ at the beginning of this subsection, i.e. the only possible combinations of classification of invariant lines are $(P_1,P_2,H),\, (P_1,H,H),\,(P_1,P_2,P_2H),\,(P_1,H,P_2H)$ and $(P_1, P_2H, P_2H)$, or in terms of Lyagina's geometric equivalence classes, we can realize classes 11, 12, 14, 15 and 16 as chemical systems. Examples are provided in Table~\ref{table1}. 

\subsection{Geometric equivalence classes with one invariant line}

\label{sec34}

In Sections~\ref{sec32} and~\ref{sec33}, we investigated the case in which there are three invariant lines~(\ref{invlines}). This occurs either when the cubic equation (\ref{cubicinvlinesigma}) has the maximum number of three real roots, as described in Lemma~\ref{lemma5}, or when the quadratic equation~(\ref{reducedcubicinvlinesigma}) has the maximum number of two real roots and the third invariant line $x=1$ corresponds to $\eta=\infty$, as described in Lemma~\ref{lemma6}. Another common case arises when the cubic equation~(\ref{cubicinvlinesigma}) has one real root and a pair of complex conjugate roots, so that the ODE system~(\ref{quadraticchemical1})--(\ref{quadraticchemical2}) has only one invariant line~(\ref{invlines}). We discuss this case in Lemma~\ref{lemma9}, followed by the remaining cases in which there is only one invariant line.

\vskip 2mm

\begin{lemma}
Let $\sigma_3 \ne 0$ and consider that the cubic equation~$(\ref{cubicinvlinesigma})$ has one real solution and a pair of two conjugate imaginary roots
$$
\eta_1 \qquad
\mbox{and}
\qquad
c_1 \pm i \, c_2
\qquad
\mbox{for} \quad c_2 \ne 0.
$$
Then the chemical triangle~$(\ref{chemicaltriangle2})$ is non-empty, i.e. ${\mathcal T} (\sigma_0,\sigma_1,\sigma_2,\sigma_3) \ne \emptyset$, if and only if one of the following two conditions holds \vskip 2mm
\noindent \rule{0pt}{1pt} \hskip 1cm {\rm (i)} $\;\;\, c_1>1$ and
\begin{equation}
\eta_1 \le \frac{2 (c_1-1)}{c_1^2+c_2^2 -1}\,,
\label{restrictioneta2g1l7a}
\end{equation}
\rule{0pt}{1pt} \hskip 0.95cm {\rm (ii)} $\;\;c_1^2+c_2^2 < c_1$ and 
\begin{equation}
\eta_1 \ge \frac{1 - c_1^2-c_2^2}{2(c_1-c_1^2-c_2^2)}\,.
\label{restrictioneta2l1l7a}
\end{equation}
\vskip 1mm
\noindent
In each case, the interior of the chemical triangle~${\mathcal T} (\sigma_0,\sigma_1,\sigma_2,\sigma_3)$ is bounded above and below by $U[\eta_1]$ and $Z[\eta_1]$, respectively.
\label{lemma9}
\end{lemma}

\vskip 2mm

\begin{proof}
Using~(\ref{chemicaltriangle}), we can follow the beginning of proof of Lemma~\ref{lemma5} to obtain that the chemical triangle~(\ref{chemicaltriangle2}) is non-empty if and only if
inequalities~(\ref{auxineq2}) are satisfied, where we substitute $\eta_2=c_1+i\,c_2$ and $\eta_3=c_1-i\,c_2$. We get
\begin{equation}
2  + \eta_1 (c_1^2+c_2^2)
\,\le\,
\eta_1 + 2 c_1
\qquad
\mbox{and}
\qquad
1 + 2 \eta_1 (c_1^2+c_2^2)
\,\le\,
2 \eta_1 c_1  + c_1^2 + c_2^2\,.
\label{auxineq2l8}
\end{equation}
If $c_1>1$, then these inequalities can be rewritten as
\begin{equation*}
\eta_1 \,\le\, \frac{2(c_1-1)}{c_1^2+c_2^2-1}
\qquad
\mbox{and}
\qquad
\eta_1 \,\le\, \frac{c_1^2 + c_2^2-1}
{2(c_1^2+c_2^2-c_1)}\,,
\end{equation*}
which reduces to~(\ref{restrictioneta2g1l7a}), because
$$
\frac{2(c_1-1)}{c_1^2+c_2^2-1}
\,\le\, \frac{c_1^2 + c_2^2-1}{2(c_1^2+c_2^2-c_1)}
\qquad
\mbox{for} \quad c_1>1 \quad
\mbox{and} \quad c_2 \in {\mathbb{R}}.
$$
Similarly, the second inequality in~(\ref{auxineq2l8}) implies~(\ref{restrictioneta2l1l7a}) if $c_1^2+c_2^2 < c_1$. Using Lemma~\ref{lemma4}, we know that the interior of the chemical triangle~${\mathcal T} (\sigma_0,\sigma_1,\sigma_2,\sigma_3)$ is bounded above by $U[\eta_1]$. Finally, the interior of the chemical triangle~${\mathcal T} (\sigma_0,\sigma_1,\sigma_2,\sigma_3)$ satisfies $\lambda > 2\sigma_3$ and $\mu > 2\sigma_0$. Substituting into equation~(\ref{zeroderline}) for $\eta=\eta_1$, we obtain the inequality
$
2\sigma_3\eta_1+2\sigma_0 \,<\, 2\eta_1(\sigma_2-\sigma_3\eta_1),
$
which, using expressions for $\sigma_0$ and $\sigma_2$ given by~(\ref{sig210expr}), can be equivalently rewritten as
$
2 \sigma_3 \eta_1 \big((c_1-1)^2+c_2^2\big) < 0,
$
which does not hold for any $c_1$, $c_2$, $\sigma_3 > 0$ and $\eta_1 \ge 0$. Consequently, the interior of the chemical triangle~${\mathcal T} (\sigma_0,\sigma_1,\sigma_2,\sigma_3)$ is bounded below by the line $Z[\eta_1]$. 
\end{proof}

\noindent
Lemma~\ref{lemma9} implies that the chemical triangle~$(\ref{chemicaltriangle2})$ lies between the lines $U[\eta_1]$ and $Z[\eta_1]$. In particular, using \cite[Proposition 4]{Newton:1978:TDH}, the invariant line corresponding to $\eta=\eta_1$ can only be classified as either $H(\eta_1)$ or $P_2H(\eta_1)$. Therefore, for a cubic equation with one real solution and a pair of two conjugate imaginary roots, the only possible combinations of classification of invariant lines are $(H)$ and $(P_2H)$, or in terms of Lyagina's geometric equivalence classes, we can only obtain classes 1 and 4~\cite{Lyagina:1951:ICE}. We follow with an example chemical system in Lyagina's class 1.

\vskip 2mm

\begin{example}
\label{example6}
Consider the following system of $m=11$ chemical reactions:
\begin{equation}
\emptyset
\;
\mathop{\longrightarrow}^{20}
\;
X \, ,
\;\quad
\emptyset
\;
\mathop{\longrightarrow}^{5}
\;
Y \, ,
\;\quad
X
\;
\mathop{\longrightarrow}^{42}
\;
\emptyset \, ,
\;\quad
Y
\;
\mathop{\longrightarrow}^{8}
\;
\emptyset \, ,
\;\quad
X
\;
\mathop{\longrightarrow}^{6}
\;
Y\, ,
\;\quad
Y
\;
\mathop{\longrightarrow}^{8}
\;
X\, ,
\;\quad
X \, + \, Y 
\;
\mathop{\longrightarrow}^{40}
\;
\emptyset \,, \;
\label{crnex6a}
\end{equation}
\begin{equation}
2 X
\;
\mathop{\longrightarrow}^{27}
\;
3X \, ,
\qquad
2 Y
\;
\mathop{\longrightarrow}^{12}
\;
3 Y \, ,
\qquad
2 X
\;
\mathop{\longrightarrow}^{17}
\;
3 X + Y
\quad
\mbox{and}
\quad
2 Y
\;
\mathop{\longrightarrow}^{16}
\;
X + 3 Y \, .
\label{crnex6b}    
\end{equation}
Then the ODEs~(\ref{quadraticchemical1})--(\ref{quadraticchemical2}) become
\begin{eqnarray}
\frac{\mbox{{\rm d}}x}{\mbox{{\rm d}}t}
\,& = &\,
20
\,-\,
48 \, x
\,+\,
8 \, y
\,+\,
44 \, x^2
\,-\,
40 \, x \, y
\,+\,
16 \, y^2 \,,
\label{quadraticchemical1ex6}
\\
\frac{\mbox{{\rm d}}y}{\mbox{{\rm d}}t}
\,& = &\,
5
\,+\,
6 \, x
\,-\,
16 \, y
\,+\,
17 \, x^2
\,-\,
40 \, x \, y
\,+\,
28 \, y^2 \,,
\label{quadraticchemical2ex6}
\end{eqnarray}
which satisfy relations~(\ref{othcoef1})--(\ref{othcoef6}). Using~(\ref{sigmavalues}), the coefficients $\sigma_i$, $i=0,1,2,3,$ are
$\sigma_0 = 17\,,$
$\sigma_1 = 84\,,$
$\sigma_2 = 68\,$
and $\sigma_3 = 16.$
The roots of the cubic polynomial~(\ref{cubicinvlinesigma}) are
$\eta_1 = 1/4$
and
$\eta_{2,3} = 2 \pm i/2$.
The phase portrait is plotted in Figure~\ref{figure3}(a), where we plot the one invariant line~(\ref{invlines}) corresponding to $\eta_1=1/4$ as the black dashed line. The ODE system~(\ref{quadraticchemical1ex6})--(\ref{quadraticchemical2ex6}) can be also rewritten in the form~(\ref{homODE1lambdamu})--(\ref{homODE2lambdamu}), where
$
\lambda = \mu = 40.
$
These values satisfy inequalities~(\ref{chemicaltriangle}), which describe the chemical triangle~(\ref{chemicaltriangle2}) in the $(\lambda,\mu)$-parameter plane. The chemical triangle is visualized as the red triangle in Figure~\ref{figure3}(b) together with the lines $Z[\eta_1]$ and $U[\eta_1]$, defined by~(\ref{zeroderline}) and~(\ref{undefline}), with their relative positioning (below and above the chemical triangle) illustrating the results of Lemma~\ref{lemma9}.
\end{example}

\vskip 2mm

\begin{figure}
\rule{0pt}{1pt}
\vskip 1mm
\rule{0pt}{1pt}
\hskip -1.5cm
\epsfig{file=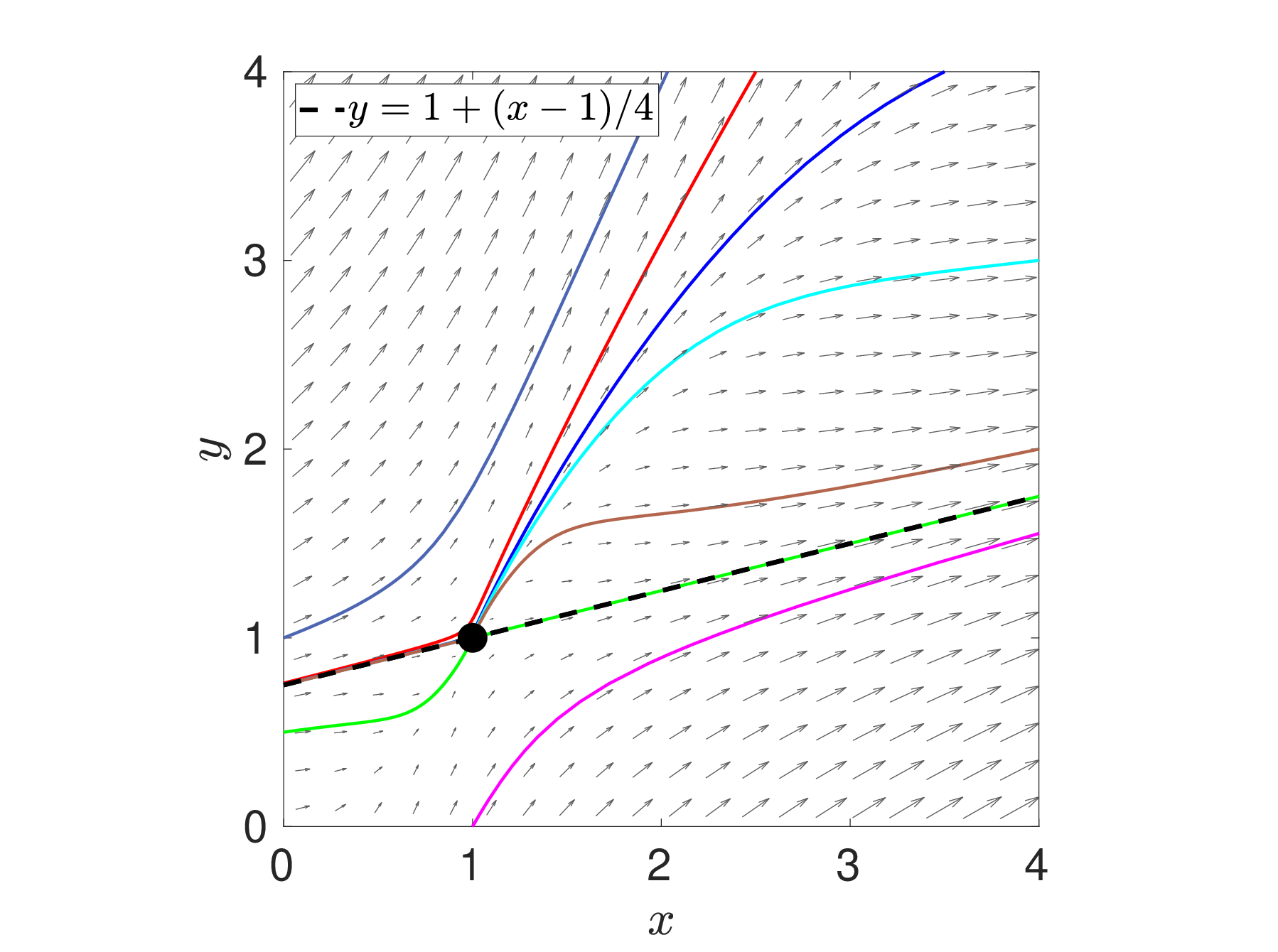,height=7.4cm}
\hskip -1.5cm
\epsfig{file=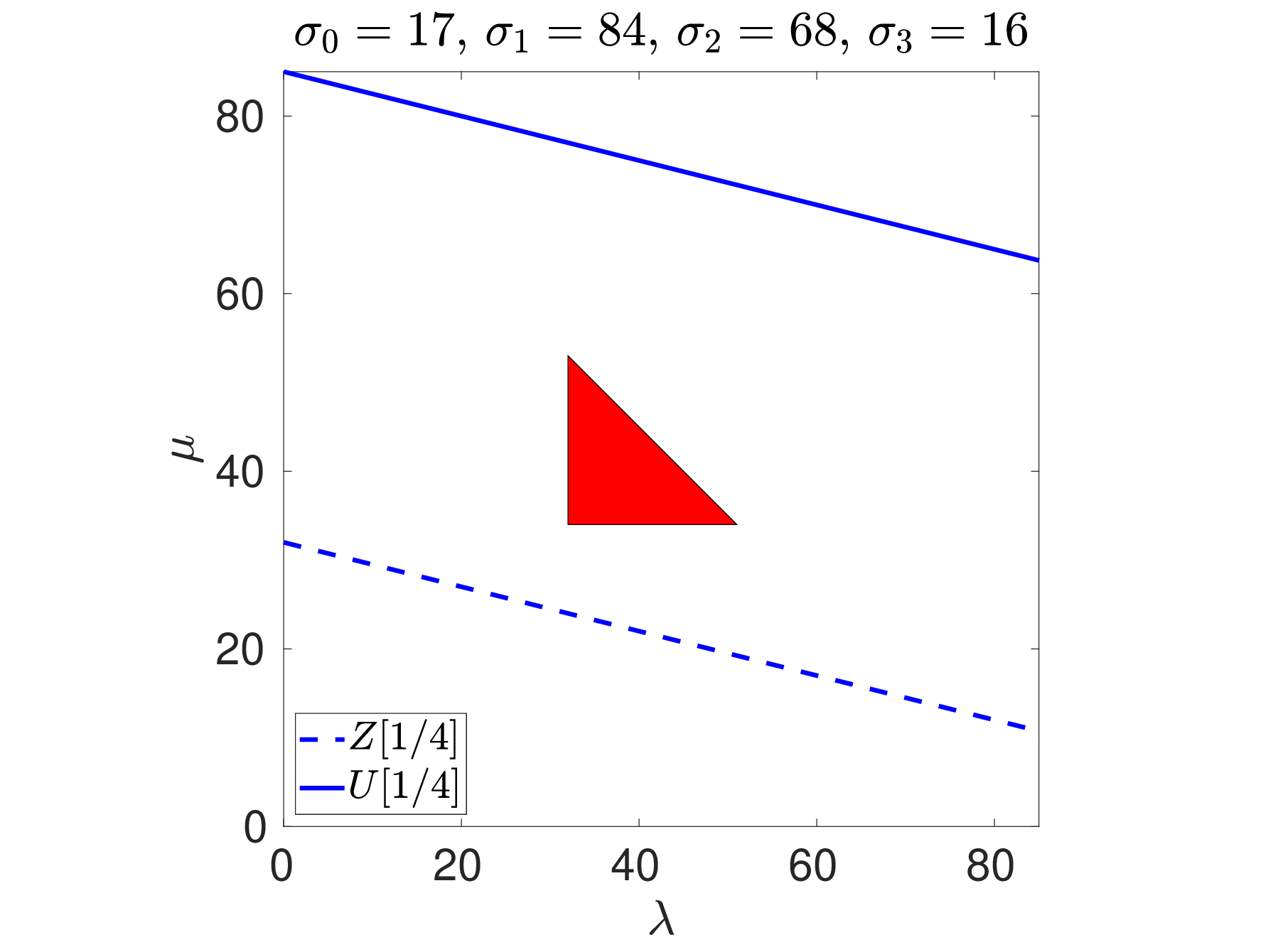,height=7.4cm}
\vskip -7.6cm
(a) \hskip 76mm (b)
\vskip 7.1cm
\caption{{\it {\rm (a)} The phase plane of the ODE system~$(\ref{quadraticchemical1ex6})$--$(\ref{quadraticchemical2ex6})$ describing the chemical reaction network~$(\ref{crnex6a})$--$(\ref{crnex6b})$. We plot seven illustrative trajectories in the non-negative quadrant $\overline{\Omega}$ together with the invariant line $y=1 +  (x-1)/4$.
\hfill\break {\rm (b)} The chemical triangle~$(\ref{chemicaltriangle})$ in the $(\lambda,\mu)$-parameter space is visualized as the red triangle together with the zero-derivative lines $Z[\eta_1]$ given by~$(\ref{zeroderline})$ and undefined-derivative lines $U[\eta_1]$ given by~$(\ref{undefline})$. 
}}
\label{figure3}
\end{figure}

\noindent
The result of Lemma~\ref{lemma9} also holds in the case $\sigma_3=0$, when the quadratic equation~(\ref{reducedcubicinvlinesigma}) has no real roots. In this case, the only invariant line is $x=1$, corresponding to $\eta=\infty$, and the interior of the chemical triangle is bounded below by the line~$Z[\infty]$ and above by the line~$U[\infty]$, defined by~(\ref{infinitycase}). This implies that the ODE systems in Lyagina's geometric equivalence classes 1 and 4 can be realized by chemical reaction systems for any $\sigma_3 \ge 0$. Examples with $\sigma_2=0$ are given in Table \ref{table1}. Another case in which there is only one invariant line occurs when the cubic equation has a triple root. This case is covered by the following lemma.

\vskip 2mm

\begin{lemma}
Let $\sigma_3 \ne 0$ and consider that the cubic equation~$(\ref{cubicinvlinesigma})$ has a triple real root $\eta$. Then the chemical triangle~$(\ref{chemicaltriangle2})$ is non-empty, i.e. ${\mathcal T} (\sigma_0,\sigma_1,\sigma_2,\sigma_3) \ne \emptyset$, if and only if $\eta = 1$. In this case, the chemical triangle reduces to a single point, i.e. 
${\mathcal T} (\sigma_0,\sigma_1,\sigma_2,\sigma_3) = \big\{ (2 \sigma_3, 2 \sigma_3) \big\}.$ 
\label{lemma8}
\end{lemma}

\begin{proof}
Using~(\ref{chemicaltriangle}), we can follow the beginning of proof of Lemma~\ref{lemma5} to obtain that the chemical triangle~(\ref{chemicaltriangle2}) is non-empty if and only if
inequalities~(\ref{auxineq2}) are satisfied, where we substitute $\eta = \eta_1 = \eta_2 = \eta_3$. We get
$$
2  + \eta^3
\,\le\,
3\eta
\qquad
\mbox{and}
\qquad
1 + 2 \eta^3
\,\le\,
3\eta^2\,
$$
which can be rewritten as
$$
(\eta-1)^2(\eta+2)\le\, 0
\qquad
\mbox{and}
\qquad
(\eta-1)^2(2\eta+1)
\,\le\,
0,
$$
which can only be true when $\eta = 1$. Thus, the chemical triangle~(\ref{chemicaltriangle2}) is non-empty if and only if $\eta = 1.$ In this case, the inequalities~(\ref{chemicaltriangle}) imply that 
$\lambda = \mu = 2 \sigma_3$, i.e. the chemical triangle ${\mathcal T} (\sigma_0,\sigma_1,\sigma_2,\sigma_3)$ reduces to the single point $(2 \sigma_3, 2 \sigma_3)$ in the $(\lambda,\mu)$-plane.
\end{proof}

\vskip 2mm

\noindent
In the case where the cubic equation~(\ref{cubicinvlinesigma}) has a triple root, Lemma~\ref{lemma8} reduces our investigation to a single chemical system with $\lambda=\mu=2\sigma_3$, $\sigma_0=\sigma_3$, and $\sigma_1=\sigma_2=3\sigma_3$. Substituting these relations into the system~(\ref{homODE1lambdamu})--(\ref{homODE2lambdamu}), we obtain the ODE system
\begin{eqnarray}
\frac{\mbox{{\rm d}}u}{\mbox{{\rm d}}t}
\,& = &\,
\sigma_3 \, (u-v)^2\,,
\label{homODE1lambdamudeg}
\\
\frac{\mbox{{\rm d}}v}{\mbox{{\rm d}}t}
\,& = &\,
\sigma_3 \, (u-v)^2 \,,
\label{homODE2lambdamudeg}
\end{eqnarray}
which describes a degenerate system with the conservation law $u-v=\mbox{constant}$ and a line of equilibria given by $u=v$. Such degenerate cases are excluded from the classification, see~\cite[Proposition 1]{Newton:1978:TDH}.

\subsection{Geometric equivalence classes with two invariant lines}

\label{sec35}

The final case to consider is when the ODE system~(\ref{homODE1lambdamu})--(\ref{homODE2lambdamu}) has exactly two invariant lines~(\ref{invlines}). If $\sigma_3\ne 0$, this occurs when the cubic equation~(\ref{cubicinvlinesigma}) has one simple real root and one double real root. This situation is investigated in the following lemma, after which we discuss the case $\sigma_3=0$.

\vskip 2mm

\begin{lemma}
Let $\sigma_3 \ne 0$ and consider that the cubic equation~$(\ref{cubicinvlinesigma})$ has a double nonnegative root $\eta_{\mathrm d}$ and a single nonnegative root $\eta_{\mathrm s}$. Then the chemical triangle~$(\ref{chemicaltriangle2})$ is non-empty, i.e. ${\mathcal T} (\sigma_0,\sigma_1,\sigma_2,\sigma_3) \ne \emptyset$, if and only if one of the following three conditions holds \vskip 2mm
\noindent \rule{0pt}{1pt} \hskip 1cm {\rm (i)} $\;\;\; \eta_{\mathrm d} = 1\,$, \vskip 1mm
\noindent \rule{0pt}{1pt} \hskip 0.95cm {\rm (ii)} $\;\;\, \eta_{\mathrm d}>1$ and
\begin{equation}
\eta_{\mathrm s}\le\frac{2}{\eta_{\mathrm d} + 1}\,,
\label{restrictioneta2g1l7}
\end{equation}
\rule{0pt}{1pt} \hskip 0.9cm {\rm (iii)} $\;\;\eta_{\mathrm d} < 1$ and 
\begin{equation}
\eta_{\mathrm s}\ge\frac{\eta_{\mathrm d}+1}{2\eta_{\mathrm d}}\,.
\label{restrictioneta2l1l7}
\end{equation}
\vskip 1mm
\noindent
In each case, the interior of the chemical triangle~${\mathcal T} (\sigma_0,\sigma_1,\sigma_2,\sigma_3)$ is bounded above and below by the lines $U[\eta_{\mathrm s}]$ and $Z[\eta_{\mathrm s}]$, respectively.
\label{lemma7}
\end{lemma}

\begin{proof}
Using~(\ref{chemicaltriangle}), we can follow the beginning of proof of Lemma~\ref{lemma5} to obtain that the chemical triangle~(\ref{chemicaltriangle2}) is non-empty if and only if
inequalities~(\ref{auxineq2}) hold, where we substitute~$\eta_1=\eta_2=\eta_{\mathrm d}$ and $\eta_3=\eta_{\mathrm s}$. We get
$2  + \eta_{\mathrm d}^2 \eta_{\mathrm s}
\,\le\,
2\eta_{\mathrm d} + \eta_{\mathrm s}$
and
$1 + 2 \eta_{\mathrm d}^2 \eta_{\mathrm s}
\,\le\,
\eta_{\mathrm d}^2
+ 2\eta_{\mathrm d}\eta_{\mathrm s}$,
which can be equivalently rewritten as
\begin{equation}
0\,\le\,(\eta_{\mathrm d}-1)\big(2-(\eta_{\mathrm d} + 1)\eta_{\mathrm s}\big)
\qquad
\mbox{and}
\qquad
0\,\le\,(\eta_{\mathrm d}-1)\big(1+\eta_{\mathrm d}-2\eta_{\mathrm d}\eta_{\mathrm s})\big).
\label{auxineq2l7}
\end{equation}
If $\eta_{\mathrm d}=1$, then both inequalities (\ref{auxineq2l7}) simplify to trivial equalities valid for all $\eta_{\mathrm s}.$ If $\eta_{\mathrm d} > 1$, then inequalities (\ref{auxineq2l7}) can be rewritten as the restriction~(\ref{restrictioneta2g1l7}) on $\eta_{\mathrm s}$ from above. If $\eta_{\mathrm d} < 1$, then inequalities (\ref{auxineq2l7}) can be rewritten as the restriction~(\ref{restrictioneta2l1l7}) on $\eta_{\mathrm s}$ from below. Using Lemma~\ref{lemma4}, we know that the interior of the chemical triangle~${\mathcal T} (\sigma_0,\sigma_1,\sigma_2,\sigma_3)$ is bounded above by $U[\eta_{\mathrm s}]$. Finally, the interior of the chemical triangle~${\mathcal T} (\sigma_0,\sigma_1,\sigma_2,\sigma_3)$ satisfies $\lambda > 2\sigma_3$ and $\mu > 2\sigma_0$. Substituting into equation~(\ref{zeroderline}) for $\eta=\eta_{\mathrm s}$, we obtain the inequality
$
2\sigma_3\eta_{\mathrm s}+2\sigma_0 \,<\, 2\eta_{\mathrm s}(\sigma_2-\sigma_3\eta_{\mathrm s}),
$
which, using expressions for $\sigma_0$ and $\sigma_2$ given by~(\ref{sig210expr}), can be equivalently rewritten as
$
2 \sigma_3 \eta_{\mathrm s} (\eta_{\mathrm d} -1)^2 < 0,
$
which does not hold for any $\sigma_3 > 0$, $\eta_{\mathrm s}\,\ge\,0$ and $\eta_{\mathrm d}\,\ge\,0$. Consequently, the interior of the chemical triangle~${\mathcal T} (\sigma_0,\sigma_1,\sigma_2,\sigma_3)$ is bounded below by the line $Z[\eta_{\mathrm s}]$. 
\end{proof}

\vskip 2mm

\noindent
Lemma~\ref{lemma7} implies that the chemical triangle~(\ref{chemicaltriangle2}) lies between $U[\eta_{\mathrm s}]$ and $Z[\eta_{\mathrm s}]$. By~\cite[Proposition 4]{Newton:1978:TDH}, the invariant line~(\ref{invlines}) corresponding to $\eta=\eta_{\mathrm s}$ can be classified as either $H(\eta_{\mathrm s})$ or $P_2H(\eta_{\mathrm s})$. Therefore, when the cubic equation has one double real root and one simple real root, the only possible combinations of invariant-line classifications are $(H,P_i-P_j)$ and $(P_2H,P_i-P_j)$. Consequently, in terms of Lyagina's geometric equivalence classes, classes 5 and 8 can be realized by chemical systems.

Lemma~\ref{lemma7} is formulated and proved for the case $\sigma_3 \ne 0$. However, the same conclusions also hold when $\sigma_3=0$. In particular, Lyagina's geometric equivalence classes 5 and 8 can also be realized by chemical systems with $\sigma_3=a_{0,2}=0$, in which case one of the two invariant lines is $x=1$, corresponding to $\eta=\infty$. We discuss this case further in the next subsection.

\subsection{Case $\sigma_3=a_{0,2}=0$ where one invariant line is $x=1$}

\label{sec36}

In Sections~\ref{sec33}, \ref{sec34}, and \ref{sec35}, we investigated three different cases in which the ODE system~(\ref{homODE1lambdamu})--(\ref{homODE2lambdamu}) has three, one, or two invariant lines, respectively. When $\sigma_3 \ne 0$, these cases correspond to the cubic equation~(\ref{cubicinvlinesigma}) having three, one, or two distinct real roots, respectively. The cases of three and one distinct real roots represent the more generic situations and are studied in Lemmas~\ref{lemma5} and \ref{lemma9}, respectively, whereas the case of two distinct real roots requires one of the roots to be a double root.

Another special case arises when $\sigma_3=0$. If $\sigma_2\ne 0$, the cubic equation~(\ref{cubicinvlinesigma}) reduces to the quadratic equation~(\ref{reducedcubicinvlinesigma}), whereas if $\sigma_2=0$ and $\sigma_1 \ne 0$, it reduces to a linear equation. When $\sigma_3=0$, one of the possible invariant lines is $x=1$, corresponding to $\eta=\infty$ in~(\ref{invlines}). In Section~\ref{sec33}, we investigated the case $\sigma_3=0$ separately in Lemma~\ref{lemma6}, whose results are consistent with the limit $\eta_3\to\infty$ of the results of Lemma~\ref{lemma5}. The same observation applies to the cases of one and two invariant lines studied in Sections~\ref{sec34} and \ref{sec35}, respectively, although we did not formulate separate lemmas for the special case $\sigma_3=0$. For example, in Section~\ref{sec35}, one could formulate results analogous to those of Lemma~\ref{lemma7}, with either the simple root $\eta_{\mathrm s}$ or the double root $\eta_{\mathrm d}$ located at $\eta=\infty$. In the former case, the quadratic equation~(\ref{reducedcubicinvlinesigma}) has a real double root, whereas in the latter case, the double root is at infinity and the cubic equation~(\ref{cubicinvline}) reduces to a linear equation, with $\sigma_3=\sigma_2=0$ and $\sigma_1 \ne 0$. In summary, the same conclusions obtained for $\sigma_3\ne 0$ also hold for $\sigma_3=0$. In particular, all nine of Lyagina's geometric equivalence classes listed in Table~\ref{table1} can be realized by chemical systems even when $\sigma_3=a_{0,2}=0$.

\begin{figure}
\epsfig{file=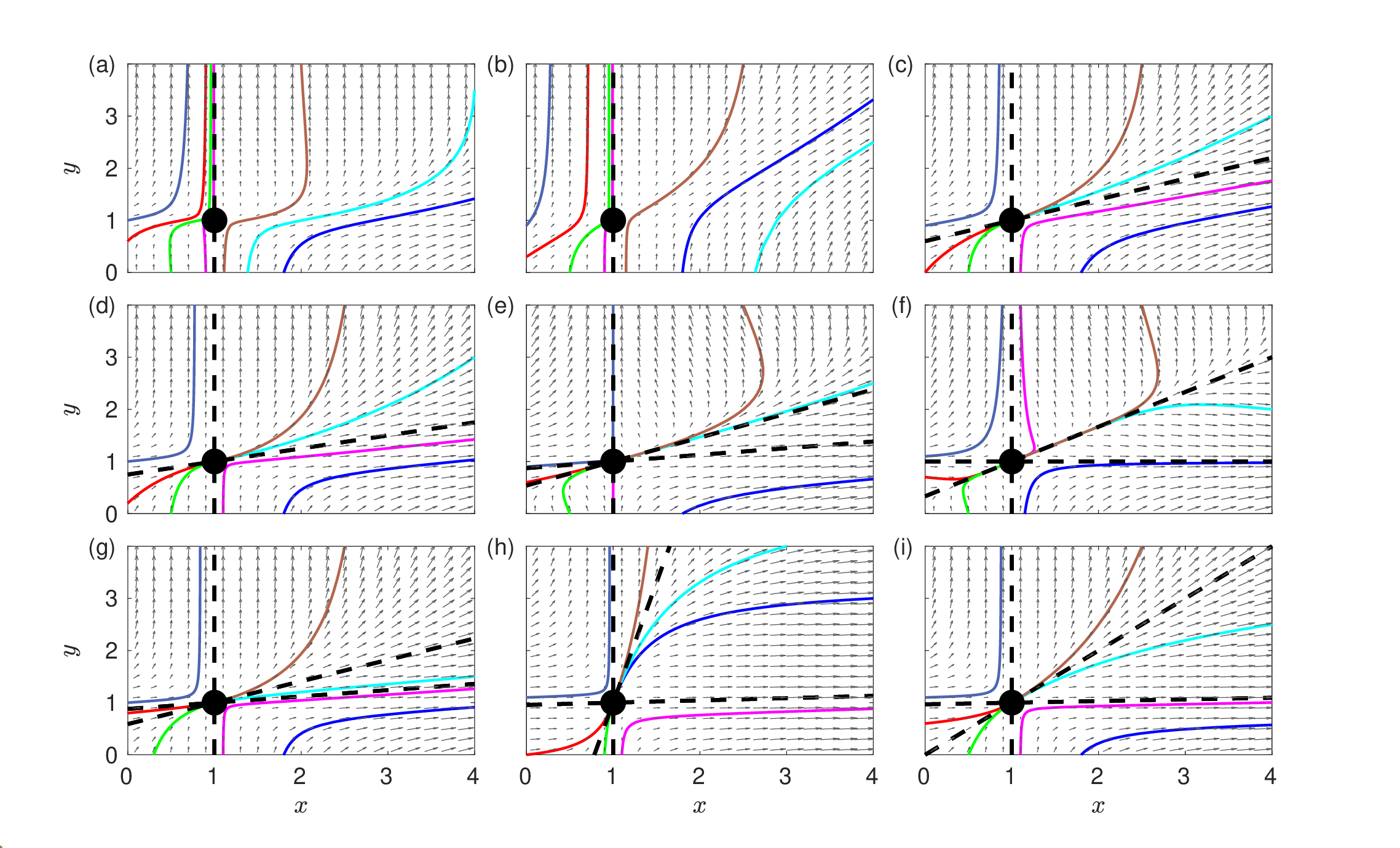,width=\textwidth}
\caption{{\it The phase planes of illustrative chemical systems given by Table~$\ref{table1}$. The invariant lines are plotted using the black dashed lines. Seven illustrative trajectories are plotted using coloured solid lines. The invariant line~$(\ref{invlines})$ corresponding to $\eta=\infty$ in plotted in each panel.}
(a) {\it Lyagina's class} $1$;
(b) {\it Lyagina's class} $4$;
(c) {\it Lyagina's class $5$, this panel includes the additional invariant line~$(\ref{invlines})$ corresponding to $\eta=2/5$};
(d) {\it Lyagina's class $8$, additional invariant line for $\eta=1/4$};
(e) {\it Lyagina's class $11$, this panel includes two additional invariant lines~$(\ref{invlines})$ corresponding to $\eta=\big(5\pm\sqrt{8}\big)/17$}; \break
(f) {\it Lyagina's class $12$, additional two invariant lines for $\eta=0$ and $\eta=2/3$};
(g) {\it Lyagina's class $14$, additional two invariant lines for $\eta=\big(43\pm\sqrt{553}\big)/162$};
(h) {\it Lyagina's class $15$, additional two invariant lines for $\eta=\big(93\pm\sqrt{8329}\big)/40$};
(i) {\it Lyagina's class $16$, additional two invariant lines for $\eta=1/32$ and $\eta=1$.}
}
\label{figure4}
\end{figure}

The chemical systems in Examples~\ref{example5} and \ref{example6} represent the most generic cases, with three and one invariant lines, respectively. The corresponding phase planes and chemical triangles are shown in Figures~\ref{figure2} and \ref{figure3} to illustrate the theory developed above. We conclude our investigation by presenting illustrative phase planes for all nine of Lyagina's geometric equivalence classes that can be realized by chemical systems. Since all nine classes can be realized by systems with $\sigma_3=0$, we focus on this case in our final set of examples, illustrated in Figures~\ref{figure4} and \ref{figure5}. These chemical systems are obtained by setting $\sigma_3=a_{0,2}=0$ and determining the remaining coefficients so that the conditions for each class are satisfied. The resulting coefficients are listed in Table~\ref{table1}. The corresponding phase planes are plotted in Figure~\ref{figure5}. In each panel, there is always the invariant line $x=1$, corresponding to $\eta=\infty$ in~(\ref{invlines}). Other invariant lines can be obtained by solving the quadratic equation~(\ref{reducedcubicinvlinesigma}), which can have zero (Figures~\ref{figure4}(a) and~\ref{figure4}(b)), one (Figures~\ref{figure4}(c) and~\ref{figure4}(d)) or two (Figures~\ref{figure4}(e)--\ref{figure4}(i)) distinct real solutions. 

For each case listed in Table~\ref{table1}, the $(\lambda,\mu)$-plane and the corresponding chemical triangle~(\ref{chemicaltriangle2}) are shown in Figure~\ref{figure5}. In each panel, we also plot the zero-derivative lines $Z[\eta_i]$, given by~(\ref{zeroderline}), and the undefined-derivative lines $U[\eta_i]$, given by~(\ref{undefline}), as dashed and solid lines, respectively. The blue lines correspond to $Z[\infty]$ and $U[\infty]$, given by~(\ref{infinitycase}), and appear in every panel. The number of additional lines is determined by the number of distinct real roots of the quadratic equation~(\ref{reducedcubicinvlinesigma}). We note that in Figures~\ref{figure5}(c) and~\ref{figure5}(d), the quadratic equation~(\ref{reducedcubicinvlinesigma}) has a double root $\eta_{\mathrm d}=\sigma_1/(2\sigma_2)$. In particular, the lines $Z[\eta_{\mathrm d}]$ and $U[\eta_{\mathrm d}]$, defined by~(\ref{zeroderline}) and~(\ref{undefline}), respectively, coincide and are given by $\eta_{\mathrm d}\lambda+\mu=\sigma_1$. Consequently, they appear as a single line in Figures~\ref{figure5}(c) and~\ref{figure5}(d).

\begin{figure}
\epsfig{file=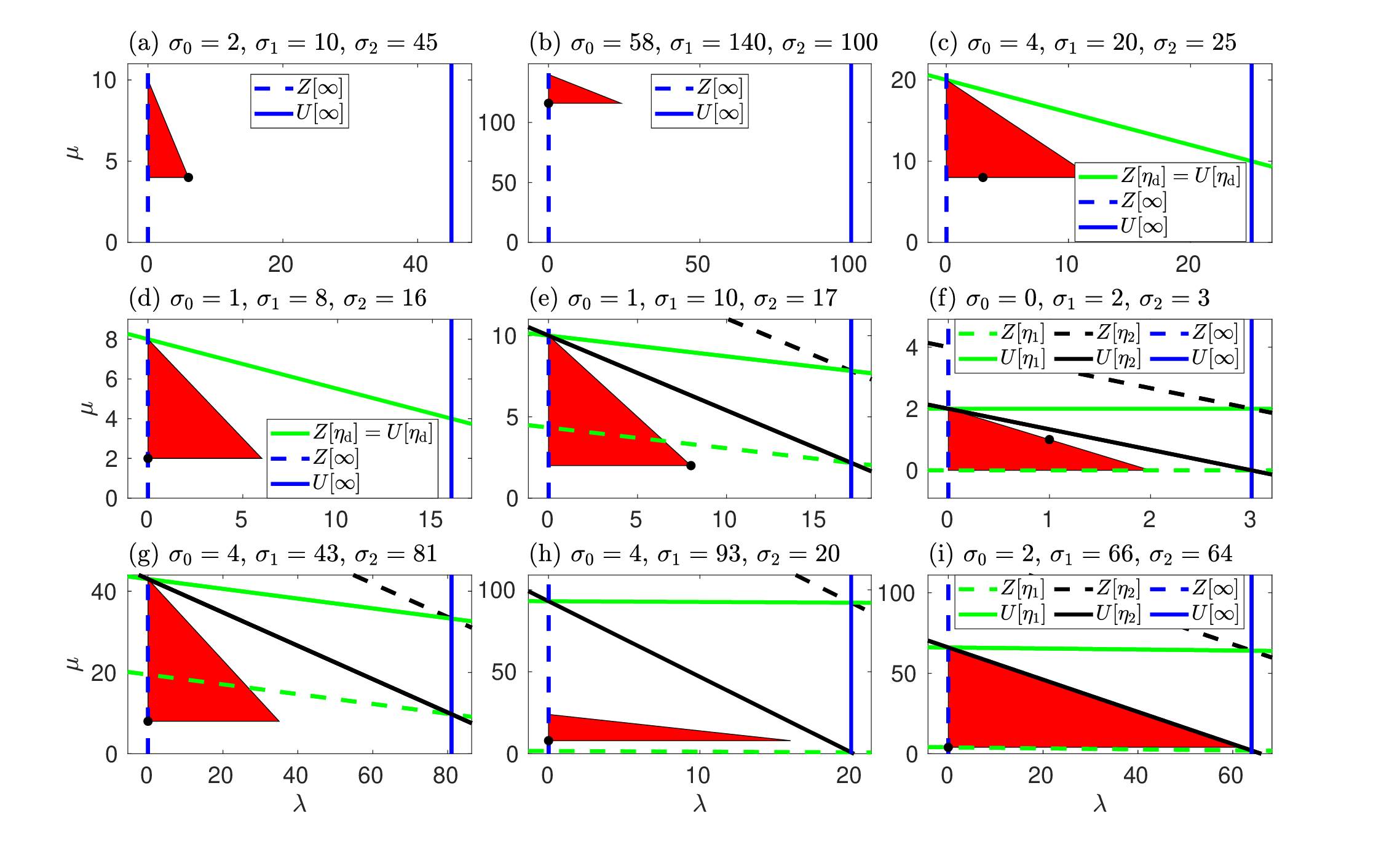,width=\textwidth}
\caption{{\it Chemical triangles~$(\ref{chemicaltriangle})$ in $(\lambda,\mu)$ planes for illustrative chemical systems given by Table~$\ref{table1}$ are plotted as red triangles. The zero-derivative lines $Z[\eta_i]$ are plotted using dashed lines and undefined-derivative lines $U[\eta_i]$ are plotted using solid lines. The black dot corresponds to the values of $(\lambda,\mu)$ used in the ODE system~$(\ref{homODE1lambdamu})$--$(\ref{homODE2lambdamu})$, which is visualized in $Figure~\ref{figure4}$ and parameterized by Table~$\ref{table1}$. The lines $Z[\infty]$ and $U[\infty]$ are plotted in each panel using blue dashed and blue solid lines, respectively.}
(a) {\it Lyagina's class} $1$;
(b) {\it Lyagina's class} $4$;
(c) {\it Lyagina's class $5$, line $Z[2/5]$ is identical to $U[2/5]$ and is visualized using green colour};
(d) {\it Lyagina's class $8$, line $Z[1/4]$ is identical to $U[1/4]$ and is visualized using green colour};
(e) {\it Lyagina's class $11$, lines $Z[\eta]$ and $U[\eta]$ are visualized for $\eta=\big(5\pm\sqrt{8}\big)/17$ using dashed and solid lines, respectively}; 
(f) {\it Lyagina's class $12$, lines $Z[\eta]$ and $U[\eta]$ are visualized for $\eta=0$ and $\eta=2/3$};
(g) {\it Lyagina's class $14$, lines $Z[\eta]$ and $U[\eta]$ are visualized for $\eta=\big(43\pm\sqrt{553}\big)/162$};
(h) {\it Lyagina's class $15$, lines $Z[\eta]$ and $U[\eta]$ are visualized for $\eta=\big(93\pm\sqrt{8329}\big)/40$};
(i) {\it Lyagina's class $16$, lines $Z[\eta]$ and $U[\eta]$ are visualized for $\eta=1/32$ and $\eta=1$.}
}
\label{figure5}
\end{figure}

\section{Discussion}

\label{sec4}

In this paper, we have studied planar second-order chemical reaction systems with a strongly non-hyperbolic equilibrium point. Since this problem is closely related to the study of systems of quadratic ODEs, for which a substantial body of theory already exists~\cite{Lyagina:1951:ICE,Newton:1978:TDH,Markus:1960:QDE}, we focused on equilibria with positive coordinates. By an appropriate rescaling of the coefficients, such an equilibrium can be assumed, without loss of generality, to be located at $[1,1]$. We have shown that, of the sixteen geometric equivalence classes identified by Lyagina~\cite{Lyagina:1951:ICE} for general quadratic ODE systems, only nine can be realized by chemical systems. In particular, we have established that phase planes containing elliptic sectors cannot be realized by planar second-order chemical reaction systems. The results are summarized in Table~\ref{table1}.

If a geometric equivalence class can be realized by a chemical reaction network, a natural question is to determine the simplest network within that class. Simplicity can be defined in several ways, for example, in terms of the smallest total number of reactions or the smallest number of second-order reactions. We have computationally searched for chemical systems with $a_{02}=0$, i.e., systems with only five second-order terms. We found that additional terms can also be eliminated and that systems with $8$–$10$ nonzero terms can be constructed for each realizable class. These examples are also given in Table~\ref{table1}. It remains an open question whether the number of nonzero terms can be reduced further. Moreover, in some cases, two nonzero terms can be realized by a single chemical reaction. For example, if $a_{1,1}=b_{1,1}=-k<0$, single chemical reaction 
$$
X \, + \, Y 
\;
\mathop{\longrightarrow}^{k}
\;
\emptyset
$$
can simultaneously generate the term $a_{1,1}\,x\,y$ in equation~(\ref{quadraticchemical1}) and the term $b_{1,1}\,x\,y$ in equation~(\ref{quadraticchemical2}). Thus, the number of chemical reactions can, in general, be smaller than the number of nonzero terms, providing an additional avenue for reducing the number of reactions in the constructed networks. For example, consider the following system of $m=5$ chemical reactions:
$$
\emptyset
\;
\mathop{\longrightarrow}^{1}
\;
Y \, ,
\qquad
Y
\;
\mathop{\longrightarrow}^{2}
\;
\emptyset \, ,
\qquad
2 X
\;
\mathop{\longrightarrow}^{1}
\;
3X \, ,
\qquad
X \, + \, Y 
\;
\mathop{\longrightarrow}^{2}
\;
Y \, 
\quad
\mbox{and}
\quad
2 Y 
\;
\mathop{\longrightarrow}^{1}
\;
X + 3 Y \, .
$$
Then the ODEs~(\ref{quadraticchemical1})--(\ref{quadraticchemical2}) have six non-zero terms
\begin{eqnarray*}
\frac{\mbox{{\rm d}}x}{\mbox{{\rm d}}t}
\,& = &\,
x^2
\,-\,
2 \, x \, y 
\,+\,
y^2\,,
\\
\frac{\mbox{{\rm d}}y}{\mbox{{\rm d}}t}
\,& = &\,
1 -2 \, y
\,+\,
y^2 \,,
\end{eqnarray*}
which satisfy relations~(\ref{othcoef1})--(\ref{othcoef6}). This system has three invariant lines passing through the equilibrium point $[1,1]$, given by $y=1+\eta_i(x-1)$, where $\eta_1=0$, $\eta_2=(3-\sqrt{5})/2$, and $\eta_3=(3+\sqrt{5})/2$. These lines divide the nonnegative quadrant $\overline{\Omega}$ into four parabolic and two hyperbolic sectors. Thus, this system provides an example of a chemical system belonging to Lyagina’s geometric equivalence class 12. It involves only $m=5$ chemical reactions, three of which are second-order reactions, and is therefore simpler than the example reported in Table~\ref{table1}, which was obtained computationally by imposing the condition $a_{0,2}=0$.

Chemical reaction systems~(\ref{crnex1}) and~(\ref{crnex2}) in Examples~\ref{example1} and~\ref{example2} include reactions with three products. Such systems cannot be {\em weakly reversible}, which is an important class of chemical reaction systems, where each reaction is contained in a directed cycle~\cite{Feinberg:2019:FCR}. If we consider such systems, then both source vertices and target vertices of the network must be at  $2X, X+Y$ or $2Y$. In particular, we conclude that the ODEs~(\ref{quadraticchemical1hom})--(\ref{quadraticchemical2hom}) obey a conservation law 
$
x(t) + y(t) \equiv c
$
for some real number $c \ge 0$, which is positive if the initial conditions are positive. While such systems are important from the point of view of the chemical reaction network theory~\cite{Feinberg:2019:FCR}, they are degenerate and not included in the classifications into geometric equivalence classes~\cite{Lyagina:1951:ICE,Newton:1978:TDH,Markus:1960:QDE}.

\vskip 6mm

\noindent
{\bf Acknowledgment:} The authors would like to express their sincere gratitude to Gheorghe Craciun for his valuable comments and suggestions during the preparation of this manuscript.

\vskip 8mm

\bibsep=0.1em

\end{document}